\documentclass[sn-mathphys-num]{sn-jnl}

\usepackage{graphicx}
\usepackage{multirow}
\usepackage[title]{appendix}
\usepackage{xcolor}
\usepackage{textcomp}
\usepackage{manyfoot}
\usepackage{booktabs}
\usepackage{array}
\usepackage{verbatim}
\usepackage{amsthm}
\usepackage{amsmath,amssymb,amsfonts,amscd,latexsym,mathrsfs}
\usepackage{stmaryrd}
\usepackage{flafter}

\usepackage{algorithm}
\usepackage{algorithmicx}
\usepackage{algpseudocode}
\usepackage{listings}
\usepackage{caption}
\usepackage{subcaption}

\usepackage[T1]{fontenc}
\usepackage[utf8]{inputenc}

\usepackage{charter} 

\usepackage{tikz-cd}

\usepackage{pifont}

\numberwithin{equation}{section}

\usepackage{hyperref} 

\newtheorem{thm}{Theorem}[section]
\newtheorem{prop}[thm]{Proposition}
\newtheorem{lem}[thm]{Lemma}

\newtheorem{conj}[thm]{Conjecture}

\newtheorem{statement}[thm]{Statement}

\newtheorem{defn}[thm]{Definition}
\newtheorem{rmk}[thm]{Remark}

\newcommand{\R}{\mathbb R}
\newcommand{\Z}{\mathbb Z}
\newcommand{\A}{\mathcal A}

\newcommand{\C}{\mathbb C}
\newcommand{\HH}{\mathbb H}
\newcommand{\G}{\mathcal G}

\newcommand{\X}{\mathcal X}

\newcommand{\F}{\mathcal F}
\newcommand{\ttt}{\mathbf t}

\newcommand{\f}{\mathbf f}

\newcommand{\g}{\mathbf g}

\begin{document}

\title[Article Title]{A Dynamical Framework for the McKay Correspondence via Gauge-Theoretic Morse Flow}


\author{\fnm{Jiajun} \sur{Yan}}\email{jy156@virginia.edu}

\affil{\orgdiv{Department of Mathematics}, \orgname{Rice University}}


\abstract{The McKay correspondence establishes a bijection between the cohomology of a minimal resolution and the irreducible representations of a finite subgroup $\Gamma \subset \text{SU}(2)$. While traditional proofs rely on static algebraic isomorphisms, we propose a dynamical framework grounded in gauge theory and Morse-Bott theory. We analyze an $S^1$-invariant Morse-Bott function on the minimal resolution, interpreting its gradient flow lines as $1$-parameter families of holonomy representations of flat connections from $\Gamma$ to $GL(R)$. We conjecture that the flow emanating from a critical submanifold converges asymptotically at the boundary to a specific irreducible representation of $\Gamma$. This dynamical process explicitly constructs the identification between the cohomology basis and the irreducible representations of $\Gamma$ prescribed by the McKay correspondence. We prove this conjecture for cyclic cases.}


\pacs[Acknowledgements]{I would like to thank Jonathon Fisher, Frances Kirwan, Ciprian Manolescu, Josh Mundinger, Steven Rayan, Richard Wentworth and  Graeme Wilkin for the helpful comments they offered on this project. I would also like to specially thank Tom Mark and Jo Nelson for their continued support throughout the project.}

\maketitle

\section{Introduction}

The McKay correspondence is first  stated in the seminal paper \cite{mckay} by McKay where he observes that there is a one-to-one correspondence between the finite subgroups $\Gamma \subset \text{SU}(2)$ and the simply laced Dynkin diagrams of types $A$, $D$, and $E$. McKay observes that decomposing the tensor product of the canonical representation with the irreducible representations of $\Gamma$ yields an adjacency matrix exactly matching that of the affine Dynkin diagram. Initially a mysterious combinatorial coincidence, this  correspondence  inspires decades of subsequent research exploring the deeper conceptual explanation for such an occurance.

This link was given a geometric foundation by Gonzalez-Sprinberg and Verdier \cite{GSV}, who lift the correspondence to the minimal resolution of the quotient singularity $X \to \mathbb{C}^2/\Gamma$. They establish that the tautological bundles on the resolution $X$ provide a $K$-theoretic basis isomorphic to the representation ring $R(\Gamma)$. Geometrically, this identified the nontrivial irreducible representations with a basis of  $H^2(X, \mathbb{Z})$. Furthermore, the intersection form of these cohomology cycles reproduces the negative Cartan matrix, thereby providing a geometric explanation for the appearance of the Dynkin diagram in the representation theory. A heuristic statement of the McKay correspondence is given below; standard references for the relevant basic topics include \cite{fulton-harris}, \cite{griffiths-harris}, and \cite{G-S}.

\begin{statement}[McKay correspondence]
Let $\Gamma \subset SU(2$) be a finite subgroup, and let $\widetilde{ \mathbb{C}^2/\Gamma}$ be the minimal resolution of $\C^2/\Gamma$. There is a natural correspondence between the non-trivial irreducible representations of $\Gamma$ and the irreducible exceptional curves in $\widetilde{ \mathbb{C}^2/\Gamma}$, such that the intersection pairing of these curves is given by the Mckay quiver of  $\Gamma$.

\end{statement}

Subsequently, more proofs of the McKay correspondence and its generalization to higher dimensions utilizing different, mainly algebriac tools and techniques, such as  \cite{B-D}, \cite{BKR}, \cite{D-L}, \cite{I-N}, \cite{I-R}, \cite{kaledin},  appear in the years to come -- the list of references provided here is not nearly exhaustive. We direct the readers to \cite{craw} and \cite{reid} for further expositions.

In the meantime, Kronheimer in his PhD thesis \cite{kronheimer}, \cite{kronheimer2}, gives a complete construction of all the $4$-dimensional hyperkähler ALE spaces, showing that each such space is topologically equivalent to the minimal resolution $X$ of the corresponding quotient $\mathbb{C}^2/\Gamma$. In Kronheimer's construction, these hyperkähler manifolds are obtained via a finite-dimensional hyperkähler reduction. Kronheimer’s construction provided $X$ with a canonical metric structure, moving the correspondence from the realm of algebraic geometry into the domain of gauge theory and metric analysis. In a subsequent paper studying the moduli space of instantons on ALE spaces \cite{kronheimer3}, joint with Nakajima, they provide a geometric interpretation of the McKay correspondence via APS-index theory, further linking the McKay correspondence with hyperkähler geometry and gauge theory. 

More recently, the McKay correspondence is explored through the lens of symplectic and contact geometry in the works of McLean-Ritter \cite{M-R} and Digiosia-Nelson \cite{D-N}. These new perspectives use symplectic cohomology and cylindrical contact homology to establish the correspondence through the enumeration of pseudoholomorphic curves. One thing to note is that the symplectic and contact geometric interpretation of the McKay correspondence does not directly address irreducible representations of $\Gamma$, and the correspondence is established between nontrivial conjugacy classes of $\Gamma$ and the generators of $H^2(X, \mathbb{Z})$.

The remarkable variety of tools, ranging from $K$-theory, Hilbert schemes, to symplectic and contact geometry, that have been used to study the McKay correspondence indicates a rich and deep link between representation theory, singularity theory and geometry. The strategy of all the existing approaches so far is to construct an isomorphism identifying bases (cohomology cycles to irreducible representations). While these celebrated results establish that an isomorphism exists, the nature of the correspondence remains largely static. A natural question, particularly in light of modern developments in gauge theory, Floer theory and mathematical physics, is whether this correspondence can be realized dynamically. Given the asymptotic metric structure of the ALE spaces constructed by Kronheimer \cite{kronheimer}, one might ask if the correspondence can manifest through the evolution of gradient flow trajectories of a natural Morse-Bott function as the flow lines approach the boundary at infinity, bearing certain reminiscence to the construction of instanton Floery homology \cite{floer}.

In this paper, we propose such a dynamical framework. The foundation of the framework lies upon a gauge-theoretic interpretation of Kronheimer's construction of the $4$-dimensional hyperkähler ALE spaces, developed in the thesis work  of the author \cite{yan}, in which each ALE space is realized as a moduli space of solutions to a system of equations for a pair consisting of a connection and a section of a vector bundle over an orbifold Riemann surface, modulo a gauge group action. We use this gauge-theoretic construction to seek a new gauge/Morse-theoretic interpretation of the McKay correspondence by studying an $S^1$-invariant Morse-Bott function on the ALE spaces and matching the gradient flow lines from its critical points with the irreducible representations of $\Gamma$. 

The setup is as follows: inspired by works of Atiyah-Bott \cite{atiyah-bott}, Hitchin \cite{hitchin} and Kirwan \cite{kirwan}, we study the topology of the moduli space via an  $S^1$-invariant Morse-Bott function induced by a Hamiltoninan $S^1$-action on the moduli space. Following Kronheimer's notations in \cite{kronheimer},  a point in the minimal resolution $X$ can be expressed as a pair of matrices $(\alpha,\beta)$ with certain properties, with $\alpha,\beta \in End(R)$ and $R$ being the regular representation of $\Gamma$. The $S^1$-action on $X$ is simply given by $$e^{i\varphi}\cdot(\alpha,\beta)=(e^{i\varphi}\alpha,e^{i\varphi}\beta).$$ The induced $S^1$-invariant Morse-Bott function is  equal to the norm square of $(\alpha,\beta)$ up to some constant, that is, $$\Phi(\alpha,\beta)=k_0(\|\alpha\|^2+\|\beta\|^2),$$ for some fixed constant $k_0$. The critical points of this $S^1$-invariant Morse-Bott function are precisely the $S^1$-fixed point of the Hamiltoninan $S^1$-action. On the other hand, from the gauge-theoretic construction in \cite{yan}, we realize each point $(\alpha,\beta)$ of an ALE space as a flat connection on a trivial bundle over $S^3/\Gamma$ with fibers equal to $R$, the regular representation of $\Gamma$. This yields a holonomy representation $\rho_{(\alpha,\beta)}: \Gamma \to GL(R)$, which can be shown to be always isomorphic to $R$. More specifically, the flat connection is given by $$d + \alpha dz_1 + \beta dz_2,$$ where $z_1,z_2$ are coordinates in $\C^2$. We set $[\alpha, \beta] = 0$ to ensure that the above connection is flat, with holonomy representation given by $$\rho_{(\alpha,\beta)}(\gamma)=R(\gamma)\exp((v_1-(uv_1-\bar{v}v_2))\alpha+(v_2-(\bar{u}v_2+vv_1))\beta),$$ where $(v_1,v_2)$ is the base point in $S^3/\Gamma$, and $u,v$ are the coordinates of an element $\gamma\in\Gamma$. We show that $\rho_{(\alpha,\beta)}$ is always isomorphic to the regular representation $R$, for any point $(\alpha,\beta)$ in $X$. 

Along a gradient flow line $(\alpha(t),\beta(t))$ from an index-$2$ critical point to infinity, we obtain a  $1$-parameter family of representations $\rho(t) = \rho_{(\alpha(t),\beta(t))}$. Since $\rho(t)$ is always isomorphic to the regular representation $R$, we obtain a $1$-parameter family of change of basis matrices $P(t)$ such that $\rho(t) = P(t) R P(t)^{-1}$. As $H^2(X)$ is generated by the components of the $S^1$-fixed points, we propose the following conjecture, presenting an identification  of the generators of $H^2(X)$ with the nontrivial irreducible representations of $\Gamma$ via the evolution of gradient flow lines, creating an explicit dynamical framework for the McKay correspondence.

\begin{conj}\label{conj}
Let $\Gamma$ be a finite subgroup of $SU(2)$. Let $X$ be the minimal resolution of $\C^2/\Gamma$.  Let $S_1$,...$S_r$ be the components of $S^1$-fixed points of $X$ generating $H^2(X)$, and let $\rho_j(t)=P_j(t)RP_j(t)^{-1}, 1\leq j\leq r$ be the $1$-parameter family of holonomy representations with the property that $$\rho_j(t)=\rho_{(\alpha_j(t),\beta_j(t))},$$ $$\lim_{t\to-\infty}(\alpha_j(t),\beta_j(t))=(\alpha_j,\beta_j),$$ with $(\alpha_j,\beta_j)\in S_j$. Then $$P_j^{\lim}=\lim_{t\to\infty}P_j(t)$$ is equal to the projector onto the nontrivial irreducible representation $\rho_j$ of $\Gamma$, and every nontrivial irreducible representation of $\Gamma$ arises this way. 
\end{conj}

In Section 6, we provide detailed verification of Conjecture \ref{conj} in low rank cyclic cases $\Z_2$, $\Z_3$. The method is indicative of the general verification for all $A_n$-type ALE spaces, which will be contained in a subsequent version of this paper. The verification of Conjecture \ref{conj} for nonabelian $\Gamma$ requires more conceptual inputs as the computation for $S^1$-fixed points becomes cumbersome in those cases. A tentative approach for future work is to realize a gradient flow line as a connection on the ALE space or on $(S^3/\Gamma) \times \mathbb{R}$ in temporal gauge, in light of the geometric interpretation of the McKay correspondence via APS-index theory \cite{kronheimer3} and the construction of instanton Floer theory \cite{floer}.

Another direction for future work is to compare this framework constructed from gauge theory and Morse theory to the symplectic and contact geometric interpretation given in \cite{M-R} and \cite{D-N}.  The robust feature of the geometry present in the McKay correspondence might potentially be an orifice to exploring structural parallel between  symplectic field theory and  gauge theory, echoing a long lasting endeavor to relate the two subfields. 

The organization of the paper is the following. Section 2 reviews Kronoheimer's construction of $4$-dimensional hyperkähler ALE spaces as finite-dimensional hyperkähler quotients and the gauge-theoretic interpretation of Kronoheimer's construction. Section 3 and Section 4 develop the main setup for this paper where we construct a flat connection from any point of an ALE space and calculate its holonomy representation. Section 5 formally introduces the $S^1$-action and  the Morse-Bott function, and analyzes the properties of the Morse-Bott function and its critical points  in detail. Section 6 proves Conjecture \ref{conj} for cyclic cases. 

\section{A review of a gauge-theoretic construction of the $4$-dimensional hyperkähler ALE spaces}

In this section, we review Kronheimer's construction \cite{kronheimer} and a parallel gauge-theoretic construction \cite{yan} of the $4$-dimensional hyperkähler ALE spaces as moduli spaces of solutions to certain gauge-theoretic equations, which will be of crucial importance to the rest of the paper. The relevance of the $4$-dimensional hyperkähler ALE spaces is such that the underlying topological spaces are precisely the minimal resolutions of $\C^2/\Gamma$ for each finite subgroup $\Gamma$ of $SU(2)$.

\begin{subsection}{Kronheimer's construction}
In this section, we give a brief review of Peter Kronheimer's original construction of ALE spaces from which the gauge-theoretic construction takes inspiration. For an introduction to moment maps and symplectic reduction, see \cite{da-silva}.

Let $\Gamma$ be a finite subgroup of $SU(2)$ and let $R$ be its regular representation. Let $Q\cong\C^2$ be the canonical $2$-dimensional representation of $SU(2)$ and let $P=Q\otimes End(R)$, where $End(R)$ denote the endomorphism space of $R$. Let $M=P^\Gamma$ be the space of $\Gamma$-invariant elements in $P$. After fixing a $\Gamma$-invariant hermitian metric on $R$, $P$ and $M$ can be regarded as right $\HH$-modules. Now, choose an orthonormal basis on $Q$, then we can write an element in $P$ as a pair of matrices $(\alpha,\beta)$ with $\alpha,\beta\in End(R)$, and the action of $J$ on $P$ is given by $$J(\alpha,\beta)=(-\beta^*,\alpha^*).$$ Since the action of $\Gamma$ on $P$ is $\HH$-linear, the subspace $M$ is then an $\HH$-submodule, which can be regarded as a flat hyperkähler manifold. Explicitly, a pair $(\alpha,\beta)$ is in $M$ if for each $$\gamma=\begin{pmatrix}
u & v \\
-v^* & u^*  
\end{pmatrix},$$ where $v^*$ and $u^*$ denote the complex conjugate of $v$ and $u$, respectively, we have 
\begin{equation}\label{1}R(\gamma^{-1})\alpha R(\gamma)=u\alpha+v\beta, \end{equation}  \begin{equation}\label{22}R(\gamma^{-1})\beta R(\gamma)=-v^*\alpha+u^*\beta.\end{equation}

Let $U(R)$ denote the group of unitary transformations of $R$ and let $F$ be the subgroup formed by elements in $U(R)$ that commute with the $\Gamma$-action on $R$. The natural action of $F$ on $P$ is given by the following: for $f\in F$, $$(\alpha,\beta)\mapsto (f\alpha f^{-1},f\beta f^{-1}).$$ 

Again, the action of $F$ on $P$ is $\HH$-linear and preserves $M$. On the other hand, since $F$ acts by conjugation, the scalar subgroup $T\subset F$ acts trivially, and hence, we get an action of $F/T$ on $M$ that preserves $I$, $J$, $K$. 

Now, let $\f/\ttt$ be the Lie algebra of $F/T$ and identify $(\f/\ttt)^*$ with the traceless elements of $\f\subset End(R)$.  As the action of $F/T$ on $M$ is Hamiltonian with respect to $I$, $J$, $K$, we obtain the following moment maps:
$$\mu_1(\alpha,\beta)=\frac{i}{2}([\alpha,\alpha^*]+[\beta,\beta^*]),$$
$$\mu_2(\alpha,\beta)=\frac{1}{2}([\alpha,\beta]+[\alpha^*,\beta^*]),$$
$$\mu_3(\alpha,\beta)=\frac{i}{2}(-[\alpha,\beta]+[\alpha^*,\beta^*]).$$

Let $\mu=(\mu_2,\mu_2,\mu_3): M\to \R^3\otimes(\f/\ttt)^*$. Let $Z$ denote the center of $(\f/\ttt)^*$ and let $\zeta=(\zeta_1,\zeta_2,\zeta_3)\in \R^3\otimes Z$. For $\zeta$ lying in the ``good set", we get $X_\zeta=\mu^{-1}(\zeta)/F$ is a smooth $4$-manifold diffeomorphic to $\widetilde{\C^2/\Gamma}$.

\begin{rmk}
We don't make the notion of the ``good set" precise here, but its definition can be found in Proposition 2.8 in \cite{kronheimer}. It is the complement of a codimension-$3$ subset of $\R^3\otimes Z$. 
\end{rmk}

\end{subsection}

\begin{subsection}{A  gauge-theoretic construction of ALE spaces}

In this section, we review a gauge-theoretic construction of all the ALE spaces parallel to Kronheimer's construction. We do so by first describing the principal and vector bundles that we will be working with for this construction. Then we define the configuration space as well as the gauge group. Finally, we will lay out the gauge-theoretic construction in more detail and state the theorem concerning the moduli space. For an introduction to gauge theory and differential geometry, see \cite{morgan}, \cite{kobayashi1} and \cite{kobayashi2}. 

\begin{subsubsection}{Step 1: Introducing the principal and vector bundles}

Let $\Gamma$ be a finite subgroup of $SU(2)$. Consider the spherical Seifert fibered space $S^3/\Gamma$. We regard $S^3/\Gamma$ as a principal $S^1$-bundle over an orbifold Riemann surface as follows: Start with $S^3$ as a principal $S^1$-bundle over $\C P^1$ via the dual Hopf fibration. The explicit construction is as follows:
$$S^3=\{(z_1,z_2)\in \C^2: |z_1|^2+|z_2|^2=1\}.$$ The $S^1$-action on $S^3$ is given by the following: for $g=e^{i\theta} \in S^1$,
$$(z_1,z_2)g=(z_1e^{-i\theta}, z_2e^{-i\theta}).$$ 

We can think of the projection map $h$ from $S^3$ to $\C P^1$ given by $h: S^3 \to \C P^1$ with $h(z_1,z_2)=[z_1:z_2]$.

The finite group $\Gamma$ acts on $S^3$ naturally and the $\Gamma$-action commutes with the $S^1$-action. Additionally, $\Gamma$ acts nontrivially on the base $\C P^1$ with fixed points. Hence, instead of a principal $S^1$-bundle over $S^2$, $S^3/\Gamma$ is a principal $S^1$-bundle over the orbifold Riemann surface $S^2/\Gamma$. 

Now, we switch to the vector bundle perspective. Let $R$ be the regular representation of $\Gamma$, and let $V=End(R)$ be the endomorphism space of $R$. Then, we can consider the associated bundle of $S^3$ with fiber $V$, denoted $E(V)$. Explicitly, $E(V)$ is defined as $E(V)=S^3 \times V/\sim$, where $[p,v]\sim[pg,g^{-1}v]$, for all $g\in S^1$. As $\Gamma$ acts on $V$, we analogously get an associated orbifold vector bundle of $S^3/\Gamma$, denoted $E(\Gamma)$. We will be working with $E(\Gamma)$ for the gauge-theoretic construction. 

\end{subsubsection}

\begin{subsubsection}{Step 2: Introducing the gauge group and the configuration space}
We see that $E(V)$ splits into a direct sum of hyperplane bundles, that is, $$E(V)=\oplus_i H_i,$$ where each $H_i$ is isomorphic to the hyperplane bundle $H$.  We also equip  $E(\Gamma)$ the following auxiliary data: We fix a $\Gamma$-invariant hermitian structure $h_V$ on $V$ and hence get pointwise metrics on $E(V)$ and $E(\Gamma)$. We take $\omega_{vol}$ to be the Fubini-Study form on $\C P^1$. We fix a holomorphic structure on $E(\C)=H$, and denote it by $\bar{\partial}$. Let $A_0$ be the unique Chern connection on $H$ compatible with the holomorphic structure $\bar{\partial}$ and the hermitian structure descending from $E(V)$. Note that $A_0$ will be $\Gamma$-invariant as it is invariant under $SU(2)$ and can be thought of as a connection on $E(\Gamma)$.

Next, we introduce the configuration space and the gauge group. Let $F\subset U(R)$ be the subgroup of unitary transformations of $R$ that commute with the $\Gamma$-action, and let $T$ be the scalar subgroup sitting inside $F$. 

\begin{defn}
 \begin{enumerate}
 \item Let the $\tau$-invariant gauge group $\G^{F,\Gamma}_\tau$ of $E(\Gamma)$ be defined as $ \G^{F,\Gamma}_\tau= Map(S^2/\Gamma,F/T)$ such that an element $\rho\in \G^{F,\Gamma}_\tau$ is invariant under the involution $\tau:\C P^1\to\C P^1,  [z_1:z_2]\mapsto [-\bar{z}_2:\bar{z}_1]$. We remark here that $\tau$ commutes with the $\Gamma$-action and hence descends to a map $\tau: S^2/\Gamma \to S^2/\Gamma$. Let $\g^{F,\Gamma}_\tau$ denote the Lie algebra of $\G^{F,\Gamma}_\tau$. We use $\rho$ to denote an element in $\G^{F,\Gamma}_\tau$, and we use $Y$ to denote an element in $\g^{F,\Gamma}_\tau$.
\item Let $\G^{F,\Gamma}_{\C,\tau}$ denote the complexification of $\G^{F,\Gamma}_\tau$, where $F^c=GL_\C(V)^\Gamma$ denotes the complex linear transformations of $R$ that commute with the $\Gamma$-action. We use $\kappa$ to denote an element in $\G^{F,\Gamma}_{\C,\tau}$.
\end{enumerate}
\end{defn}

\begin{defn}
We define the configuration space to be $\A^F_\tau \times C^\infty(S^2/\Gamma, E(\Gamma))$ where $\A^F_\tau$ and $C^\infty(S^2/\Gamma, E(\Gamma))$ are defined as follows.
\begin{enumerate}
\item Let $\A^F_\tau$ be the space of connections on $E(\Gamma)$ given by $$\A^F_\tau=\{A_0+\kappa^*\partial \kappa^{*-1}+\kappa^{-1}\bar{\partial}\kappa \mid  \kappa \in \G^{F,\Gamma}_{\C,\tau} \},$$ where $A_0$ is the aforementioned Chern connection on $H$ or equivalently $S^3$ thought of as the induced connection on $E(\Gamma)$. We will always denote $\kappa^*\partial \kappa^{*-1}+\kappa^{-1}\bar{\partial}\kappa$ by $B$, and sometimes we omit the base connection $A_0$.
\item Let $C^\infty(S^2/\Gamma, E(\Gamma))$ denote the abbreviation for the space of orbifold vector bundle sections $C^\infty_{orb}(S^2/\Gamma, E(\Gamma))$. 
\end{enumerate}

\end{defn}

Notice that $\G^{F,\Gamma}_\tau$ is the subgroup of the group of unitary gauge automorphisms of $E(\Gamma)$. The action of $\rho \in \G^{F,\Gamma}$ on $\A^F_\tau \times C^\infty(S^2/\Gamma, E(\Gamma)) $ is given by the following: for a pair $(B,\Theta) \in \A^F_\tau \times C^\infty(S^2/\Gamma, E(\Gamma))$, $$\rho \cdot (B,\Theta)=(B+\rho d_{B}\rho^{-1},\rho \Theta \rho^{-1}).$$ Note that here we omit the base connection $A_0$ as $\rho$ fixes $A_0$.

\end{subsubsection}

\begin{subsubsection}{Step 3: A sketch of the gauge-theoretic construction of ALE spaces and the main theorem}

In this subsection, we describe the main construction and state the main theorem. To start off, we write down the standard symplectic structure on the configuration space.

\begin{defn}[Symplectic structure on $\A^F_\tau \times C^\infty(S^2/\Gamma, E(\Gamma))$]
Let $$(B_1,\Theta_1), (B_2,\Theta)\in\A^F_\tau \times C^\infty(S^2/\Gamma, E(\Gamma)).$$  Let a symplectic $2$-form $\boldsymbol{\Omega}$ on $\A^F_\tau \times C^\infty(S^2/\Gamma, E(\Gamma))$ be defined as follows: $$\boldsymbol{\Omega}((B_1,\Theta_1), (B_2,\Theta))=\int_{S^2/\Gamma}B_1\wedge B_2+\int_{S^2/\Gamma}-Im\langle\Theta_1,\Theta_2\rangle\omega_{vol}.$$ 
\end{defn}

More interestingly, in addition to the standard symplectic structure, the space of sections $C^\infty(S^2/\Gamma, E(\Gamma))$ is in fact hyperkähler. Before we write down the kähler forms, we first introduce some notations. 

For a section $\Theta$ on $C^\infty(S^2/\Gamma, E(\Gamma))$, we identify $\Theta$ with an $S^1$- and $\Gamma$-equivariant map $\lambda: S^3 \to End(R)$, and hence we can express $\Theta$ as $$\Theta: x \mapsto [p, \lambda(p)],$$ for $x \in S^2$ and $p \in h^{-1}(x) \subset S^3$.

There is a complex structure $J$, in addition to the standard complex structure $I$ given by the multiplication by $i$, on the space of sections $C^\infty(S^2/\Gamma, E(\Gamma))$, which we can express as follows. Let $\Theta: x \mapsto [p, \lambda(p)]$ be given, the action of $J$ on $\Theta$ is given by $$J\Theta: x \mapsto [p, -\lambda(J(p))^*],$$ where $p \in S^3$ and $J$ on $S^3$ is just the usual quaternion action. 

\begin{prop}
There are three symplectic forms on $C^\infty(S^2/\Gamma, E(\Gamma))$ compatible with complex structures $I$, $J$, $K$, respectively:

$$\omega_1(\Theta_1,\Theta_2)=\int_{S^2/\Gamma}-Im\langle \Theta_1,\Theta_2\rangle\omega_{vol},$$
$$\omega_2(\Theta_1,\Theta_2)=\int_{S^2/\Gamma}Re \langle J\Theta_1, \Theta_2\rangle \omega_{vol},$$
$$\omega_3(\Theta_1,\Theta_2)=\int_{S^2/\Gamma}-Im\langle J\Theta_1, \Theta_2\rangle\omega_{vol},$$
and a hyperkähler metric $g_h$ such that
$$g_h(\Theta_1,\Theta_2)=\int_{S^2/\Gamma}Re\langle\Theta_1,\Theta_2\rangle\omega_{vol},$$
together giving rise to a hyperkähler structure on $C^\infty(S^2/\Gamma, E(\Gamma))$.

\end{prop}

It turns out that the action of the $\tau$-invariant gauge group $\G^{F,\Gamma}_\tau$ on the space of sections of $E(\Gamma)$ with respect to each one of the three symplectic forms is Hamiltonian. Hence, we can write down the following equations inspired by the hyperkähler reduction construction and operate a reduction on the configuration space given by the main theorem stated below:

For $\tilde{\zeta}=(\tilde{\zeta}_1,\tilde{\zeta}_2,\tilde{\zeta}_3) \in \R^3\otimes Z$, where $Z$ is the center of $(\f/\ttt)^*$ thought of as traceless matrices in $\f/\ttt$, we consider

\begin{equation} \label{holo}
\bar{\partial}_{A_0+B} \Theta=0
\end{equation}

\begin{equation} \label{eq1}
F_B-\frac{i}{2}[\Theta,\Theta^*]\omega_{vol}=\tilde{\zeta}_1\cdot\omega_{vol}
\end{equation}

\begin{equation}\label{mu2}
-\frac{1}{4}([J\Theta,\Theta^*]-[\Theta,J\Theta^*])\omega_{vol}=\tilde{\zeta}_2\cdot\omega_{vol},
\end{equation}

\begin{equation}\label{mu3}
-\frac{i}{4}([J\Theta,\Theta^*]+[\Theta,J\Theta^*])\omega_{vol}=\tilde{\zeta}_3\cdot\omega_{vol}.
\end{equation}

\begin{thm}\label{thm}
Let $\tilde{\zeta}=(\tilde{\zeta}_1,\tilde{\zeta}_2,\tilde{\zeta}_3)$, where for all $i$, $\tilde{\zeta}_i\in Z$. Let $$\X_{\tilde{\zeta}}=\{(B,\Theta)\in \A^F \times C^\infty(S^2/\Gamma, E(\Gamma)) | (1) - (4)\}/\G_\tau^{F,\Gamma}.$$ Then for a suitable choice of $\tilde{\zeta}$, $\X_{\tilde{\zeta}}$ is diffeomorphic to the resolution of singularity $\widetilde{\C^2/\Gamma}$. Furthermore, for $\zeta=\tilde{\zeta}^*=-\tilde{\zeta}$, there exists a map $\Phi$ taking  $X_{\zeta}$ to $\X_{\tilde\zeta}$ and a natural choice of metric on $\X_{\tilde\zeta}$  such that $\Phi$ is an isometry.
\end{thm}

\begin{rmk}
We don't make the notion of ``a suitable choice of $\tilde{\zeta}$" in the statement of the theorem precise here. It is defined analogously to the notion of the ``good set" appearing in the previous section, introduced by Kronheimer. Again, it is the complement of a codimension-$3$ subset of $\R^3\otimes Z$. 
\end{rmk}

\end{subsubsection}

\end{subsection}

\section{Hopf fibration and the pullback bundle $E(\Gamma)^h$}

As introduced in the previous section, the orbifold vector bundle $E(\Gamma)$ is built from the Hopf fibration of $S^3$. In this section, we define a pullback bundle $E(\Gamma)^h$ of $E(\Gamma)$ by the Hopf fibration $h: S^3 \to \C P^1$. First, we pullback $E(End(R))$ by $h$ and get the following commutative diagram:  \begin{center}
\begin{tikzcd}

E(End(R))^h \arrow{r} \arrow{d}{\pi^h}&  E(End(R)) \arrow{d}{\pi}\\
S^3 \arrow{r}{h}& S^2\\
\end{tikzcd}
\end{center}

The above diagram is $\Gamma$-equivariant and hence descends to the following commutative diagram:

\begin{center}
\begin{tikzcd}

E(\Gamma)^h \arrow{r} \arrow{d}{\pi^h}&  E(\Gamma) \arrow{d}{\pi}\\
S^3/\Gamma \arrow{r}{h}& S^2/\Gamma\\
\end{tikzcd}
\end{center}

\begin{lem}
$E^h(\Gamma) $ is a trivial vector bundle over $S^3/\Gamma$ with fibers isomorphic to $End(R)$.
\end{lem}

We omit the proof of the previous lemma as it is a standard result in bundle theory applied to the particular setting. We think of $E^h(\Gamma) \cong S^3/\Gamma \times End(R)$ as the bundle of endomorphisms acting on the trivial bundle $S^3/\Gamma \times R$. Let $\Theta$ be a section of $E(\Gamma)$, then by identifying $\Theta$ with a $\Gamma$-equivariant map $\lambda: S^3\to End(R)$, we can think of $\Theta$ as a section of $E^h(\Gamma)$ as well. As a result, we can define $A_\Theta$ as follows: $$A_\Theta=d + d\Theta.$$ We think of $A_\Theta$ as a connection acting on the sections of  $S^3/\Gamma \times R$. In particular, if $\Theta$ is a holomorphic section and we write $\Theta=(\alpha,\beta)$, then we have $$A_\Theta=A_{(\alpha,\beta)}=d+\alpha dz_1+\beta dz_2.$$

\begin{lem}
Let $\Theta=(\alpha,\beta)$ be a holomorphic section. The connection $A_\Theta$ is flat if and only if $\zeta_2=\zeta_3=0$.
\end{lem}

\begin{proof}
We compute the curvature of $A_\Theta$ as follows: \begin{align}F_{A_\Theta}&=dA_\Theta+ A_\Theta\wedge A_\Theta \\&=\alpha dz_1\wedge \beta dz_2+\beta dz_2\wedge \alpha dz_1\\&=[\alpha,\beta]dz_1\wedge dz_2. \end{align} Hence, $F_{A_\Theta}=0$ is equivalent to $[\alpha,\beta]=0$. By equations $$\frac{1}{2}([\alpha,\beta]+[\alpha^*,\beta^*])=\zeta_2,$$
$$\frac{i}{2}(-[\alpha,\beta]+[\alpha^*,\beta^*])=\zeta_3,$$ we have $[\alpha,\beta]=0$ is equivalent to $\zeta_2=\zeta_3=0$.

\end{proof}

\section{Calculation of the holonomy representations}

From this point on, we fix $\zeta=(\zeta_1,0,0)$ to be an element in the good set. Note that this is always possible since the good set is the complement of a codimension-$3$ subset of $\R^3\otimes Z$, and by setting $\zeta_2=\zeta_3=0$, we can choose $\zeta_1$ to lie in a codimension-$1$ subset of $Z$ to ensure $\zeta=(\zeta_1,0,0)$ lies in the good set. Continuing from the previous section, since $\zeta_2=\zeta_3=0$, a solution $\Theta=(\alpha,\beta)$ gives rise to a flat connection $A_\Theta$ and hence induces a holonomy representation $$\rho_\Theta: \pi_1(S^3/\Gamma)=\Gamma \to GL(R). $$
We now compute $\rho_\Theta$. First, fix $(v_1,v_2)\in S^3\subset \C^2$ and $\gamma\in\Gamma$, let $$\tilde{\gamma}:[0,1]\to \C^2$$ be the linear path connecting $(v_1,v_2)$ to $\gamma\cdot(v_1,v_2)$. More explicitly, we let $$\gamma=\begin{pmatrix}
u & v \\
-v^* & u^* 
\end{pmatrix}\in\Gamma\subset SU(2),$$ and we have $\tilde{\gamma}(t)=((1-t+tu)v_1-t\bar{v}v_2,tvv_1+(1-t+t\bar{u})v_2)$, with $\gamma$ acting on $\C^2$ on the right. Now we define the path $\gamma(t)$ to be the projection of $\tilde{\gamma}(t)$ onto $S^3$, that is $$\gamma(t)=\frac{\tilde{\gamma}(t)}{\Vert \tilde{\gamma}(t) \Vert}.$$ 

By construction, $\gamma(t)$ is a loop representing $\gamma\in\pi_1(S^3/\Gamma;(v_1,v_2))=\Gamma$. Now, to calculate the holonomy representation at $A_\Theta$, we need to solve the following differential equation given by the parallel transport along $\gamma(t)$: \begin{equation} dX(t)_{{\gamma}(t)}({\dot{\gamma}(t)})+ \alpha(X(t)){dz_1}_{{\gamma}(t)}({\dot{\gamma}(t)})+\beta(X(t)){dz_2}_{{\gamma}(t)}({\dot{\gamma}(t)})=0,\end{equation} where $X(t):[0,1]\to R$, $X(0)=X_0$. Here, we can think of $X(t)$ as $X(t)=X\vert_{\gamma(t)}$, where $X$ is a section of the trivial bundle over $S^3$ with fibers isomorphic to $R$. Equivalently, we can write equation (4.1) as \begin{equation} dX(t)_{{\gamma}(t)}({\dot{\gamma}(t)})=- \alpha(X(t)){dz_1}_{{\gamma}(t)}({\dot{\gamma}(t)})-\beta(X(t)){dz_2}_{{\gamma}(t)}({\dot{\gamma}(t)}).\end{equation} Equation (4.2) is an ordinary differential equation, and at $\gamma(t)=(z_1,z_2)$, we have $$X(t)=X(z_1,z_2)=\exp((v_1-z_1)\alpha+(v_2-z_2)\beta)(X_0).$$ We can check that \begin{align*}
dX_{(z_1,z_2)} &= (-\alpha\exp((v_1-z_1)\alpha+(v_2-z_2)\beta)dz_1)(X_0)\\&+(-\beta\exp((v_1-z_1)\alpha+(v_2-z_2)\beta)dz_2)(X_0) \\
  &= -\alpha(X(z_1,z_2))dz_1-\beta(X(z_1,z_2))dz_2.
\end{align*}
 Note that for $X(t)$ to be well-defined, we need $[\alpha,\beta]=0$, which is precisely when $A_\Theta$ is flat. Hence, we have $X(v_1,v_2)=X(0)=X_0$, and $$X(1)=X(\gamma(v_1,v_2))=\exp((v_1-(uv_1-\bar{v}v_2))\alpha+(v_2-(\bar{u}v_2+vv_1))\beta)(X_0).$$ To get the holonomy representation at $A_\Theta$, we pullback $X(1)$ to the same fiber as $X_0$ over $(v_1,v_2)$ by the regular representation of $\Gamma$. As a result, we have the holonomy representation $\rho_\Theta=\rho_{(\alpha,\beta)}: \Gamma \to GL(R)$ given by $$\gamma \mapsto R(\gamma)\exp((v_1-(uv_1-\bar{v}v_2))\alpha+(v_2-(\bar{u}v_2+vv_1))\beta).$$

\begin{prop}\label{prop}
$\rho_\Theta$ is isomorphic to the regular representation of $\Gamma$, for any $\Theta$ in $X_\zeta$.
\end{prop}

We postpone the proof of Proposition \ref{prop} to the end of Section 5 after we introduce an $S^1$-action on the moduli space $X_\zeta$.

\section{An $S^1$-invariant Morse-Bott function on the moduli space and its critical points}

In this section, we introduce an $S^1$-action on the configuration space that descends to the moduli space $X_\zeta$ and analyze an $S^1$-invariant Morse-Bott function on the moduli space arising from this action. For references related to these topics, see \cite{kirwan} and \cite{fisher}. The $S^1$-action we consider is as follows: $$e^{i\varphi}: (A_0+B,\Theta)\mapsto (A_0+B, e^{i\varphi}\Theta).$$ This $S^1$-action is similar to the one considered by Hitchin. It is the natural $S^1$-action to consider in our setting as equations (2.3)--(2.6) are invariant under this action once we set $\zeta_2=\zeta_3=0.$ 

This $S^1$-action is Hamiltonian and hence gives rise to a Morse-Bott functional on the configuration space as the norm square of the moment map. We  write down a Morse-Bott functional on the configuration space as follows: $$\Phi(A_0+B,\Theta)=\int_{S^2/\Gamma}\lVert \Theta\rVert^2 \omega_{vol}.$$ We observe that this Morse-Bott functional is $S^1$-invariant; restricted on the set of solutions to the equations defining the moduli space, and up to gauge transformations, its critical points are given by the following equations: 

\begin{equation}  \rho(\varphi)(A_0+B)\rho(\varphi)^{-1}=A_0+B, \end{equation} 
\begin{equation} \rho(\varphi)\Theta\rho(\varphi)^{-1}=e^{i\varphi}\Theta, \forall \varphi. \end{equation}

Here, $(A_0+B,\Theta)$ is a solution, and $\{\rho(\varphi)\}$ is an $S^1$-subgroup of gauge transformations in $\G^\F_\tau$. Notice that since $\rho(\varphi)$ acts trivially on $A_0$, equation (5.1) reduces to \begin{equation}  \rho(\varphi)B\rho(\varphi)^{-1}=B. \end{equation} 
Furthermore, for $(A_0+B,\Theta)$ to be a solution, $B$ must be gauge-equivalent to $0$. We have that the based gauge group $\G^\F_{\tau,0}$ acts freely on $B$ and the stabilizer of $B$ is isomorphic to $F$. Hence, equations (5.1) and (5.2) can be reduced to the following:
\begin{equation}  f(\varphi)\Theta f(\varphi)^{-1}=e^{i\varphi}\Theta, \forall \varphi, \end{equation} where $\{f(\varphi)\}$ is an $S^1$-subgroup of $F$, and $\Theta$ is a holomorphic section of $E(\Gamma)$ with respect to $A_0$. In other words, after identifying $\Theta$ with $(\alpha,\beta)$, equation (5.4) becomes $$(f(\varphi)\alpha f(\varphi)^{-1},f(\varphi)\beta f(\varphi)^{-1})=(e^{i\varphi}\alpha,e^{i\varphi}\beta), \forall \varphi.$$

\begin{lem}
For a holomorphic sections $\Theta=(\alpha,\beta)$, $\Phi(\Theta)=\int_{S^2/\Gamma}\lVert \Theta\rVert^2 \omega_{vol}=k_0(\lVert \alpha\rVert^2 +\lVert \beta\rVert^2)$, for some constant $k_0$.
\end{lem}

\begin{proof}
For $p=(z_1,z_2)\in S^3$, we have $\Theta(h(p))=\alpha z_1+\beta z_2$. Hence, we have \begin{align}f(\Theta)&= \frac{1}{\text{Vol}(S^1)}  \int_{S^3/\Gamma} \| \alpha z_1 + \beta z_2 \|^2 \, \omega_{\text{vol}_{S^3/\Gamma}} \\&=\frac{1}{\text{Vol}(S^1)} ( \|\alpha\|^2 \int_{S^3/\Gamma} |z_1|^2 + \|\beta\|^2 \int_{S^3/\Gamma} |z_2|^2+\langle \alpha, \beta \rangle\int_{S^3/\Gamma} z_1 \bar{z}_2+ \langle \beta, \alpha \rangle
\int_{S^3/\Gamma} z_2 \bar{z}_1 )\\&=k_0(\lVert \alpha\rVert^2 +\lVert \beta\rVert^2). \end{align}

\end{proof}
\begin{lem}
Let $\zeta=(\zeta_1,0,0)$ be an element in the good set. The functional $\Phi$ descends to a perfect Morse-Bott function on the moduli space $X_\zeta$. 
\end{lem}

\begin{proof}
It suffices to show that the scaling $S^1$-action is proper and extends to a holomorphic $\C^*$-action on the moduli space, which guarantees perfectness. 

Since $S^1$ is compact and the action is Hamiltonian with the Morse-Bott function given by the norm square of the moment map, properness follows. As we use the scaling $S^1$-action, we can extend the action to the following complex linear action, which is holomorphic: $$z\cdot(\alpha,\beta)=(z\alpha,z\beta), z\in\C^*.$$ As a result, we get perfectness.

\end{proof}

We proceed to analyze the gradient flow lines of $\Phi$ on $X_\zeta$. On the configuration space, the gradient of $\Phi$ at $(\alpha,\beta)$ is simply $$\nabla \Phi(\alpha,\beta)=\nabla(\|\alpha\|^2+\|\beta\|^2)=\nabla(\alpha\alpha^*+\beta\beta^*)=(2\alpha,2\beta).$$ In the quotient under gauge equivalence, the gradient of $\Phi$ is given by $$\nabla \Phi(\alpha,\beta)=(2\alpha+[\xi,\alpha],2\beta+[\xi,\beta]),$$ for some $\xi\in \f/\ttt$. Hence,  a gradient flow line $(\alpha(t),\beta(t))$ in $X_\zeta$ is defined by the following equations: \begin{equation}2\alpha(t)+[\xi(t),\alpha(t)]=\dot{\alpha}(t), \end{equation}
\begin{equation}2\beta(t)+[\xi(t),\beta(t)]= \dot{\beta}(t), \end{equation}
\begin{equation} [\alpha(t),\alpha(t)^*]+[\beta(t),\beta(t)^*]=\zeta_1, \end{equation}
\begin{equation} [\alpha(t),\beta(t)]=0.\end{equation}

\begin{rmk}
Note that in (5.7), for simplicity, we omit the constant $\frac{i}{2}$ present in $\mu_1(\alpha,\beta)=\frac{i}{2}([\alpha,\alpha^*]+[\beta,\beta^*]),$ which implies that $\zeta_1$ in (5.7) is a real-valued matrix instead. We will keep this convention going forward. 
\end{rmk}

\begin{lem}
Let $(\alpha,\beta)$ be a critical point and suppose $(\alpha(t),\beta(t))$ is a gradient flow line into $(\alpha,\beta)$, as $t \to -\infty$. Then $(e^{i\varphi}\alpha(t),e^{i\varphi}\beta(t))$ is a gradient flow line into $(\alpha,\beta)$, as $t \to -\infty$ , for all $\varphi$. \end{lem}

\begin{proof}
Since $(\alpha(t),\beta(t))$ is a gradient flow line into $(\alpha,\beta)$, we have the following equations satisfied for all $t$:
\begin{equation}2\alpha(t)+[\xi(t),\alpha(t)]=\dot{\alpha}(t), \end{equation}
\begin{equation}2\beta(t)+[\xi(t),\beta(t)]= \dot{\beta}(t), \end{equation}
\begin{equation} [\alpha(t),\alpha(t)^*]+[\beta(t),\beta(t)^*]=\zeta_1, \end{equation}
\begin{equation} [\alpha(t),\beta(t)]=0,\end{equation}
\begin{equation}\lim_{t\to-\infty}(\alpha(t),\beta(t))=(\alpha,\beta).\end{equation}
It's clear that from the above equations, $(e^{i\varphi}\alpha(t),e^{i\varphi}\beta(t))$ is another gradient flow line. Since $(\alpha,\beta)$ is a critical point, it is an $S^1$-fixed point up to gauge equivalence, that is, $(e^{i\varphi}\alpha,e^{i\varphi}\beta)=(f(\varphi)\alpha f(\varphi)^{-1},f(\varphi)\beta f(\varphi)^{-1})$. We have $$\lim_{t\to-\infty}(e^{i\varphi}\alpha(t),e^{i\varphi}\beta(t))=(e^{i\varphi}\alpha,e^{i\varphi}\beta)=(f(\varphi)\alpha f(\varphi)^{-1},f(\varphi)\beta f(\varphi)^{-1}).$$ Hence, $(e^{i\varphi}\alpha(t),e^{i\varphi}\beta(t))$ is a gradient flow line into $(\alpha,\beta)$ in the moduli space $X_\zeta$.
\end{proof}

\subsection{ $S^1$-fixed points of $A_n$-type ALE spaces}

In this subsection, we specialize to the $A_n$-type ALE spaces, that is, $\widetilde{\C^2/\Gamma}$ for a cyclic subgroup $\Gamma$ of $SU(2)$. We give a detailed description of the $S^1$-fixed points of this type of ALE spaces with respect to the scaling $S^1$-action $$e^{i\varphi}\cdot(\alpha,\beta)=(e^{i\varphi}\alpha,e^{i\varphi}\beta).$$ 

To start, we recall from (\ref{1}) and (\ref{22}), for $(\alpha,\beta)\in M$,  and  $\gamma=\begin{pmatrix}
u & 0 \\
0 & u^*  
\end{pmatrix}\in \Gamma,$  we have 
\begin{equation}R(\gamma^{-1})\alpha R(\gamma)=u\alpha, \end{equation}  \begin{equation}R(\gamma^{-1})\beta R(\gamma)=u^*\beta.\end{equation}

It's not hard to see that combinatorially, such a pair $(\alpha,\beta)$ is represented by the following diagram: \begin{center}
\includegraphics[width=0.5\textwidth]{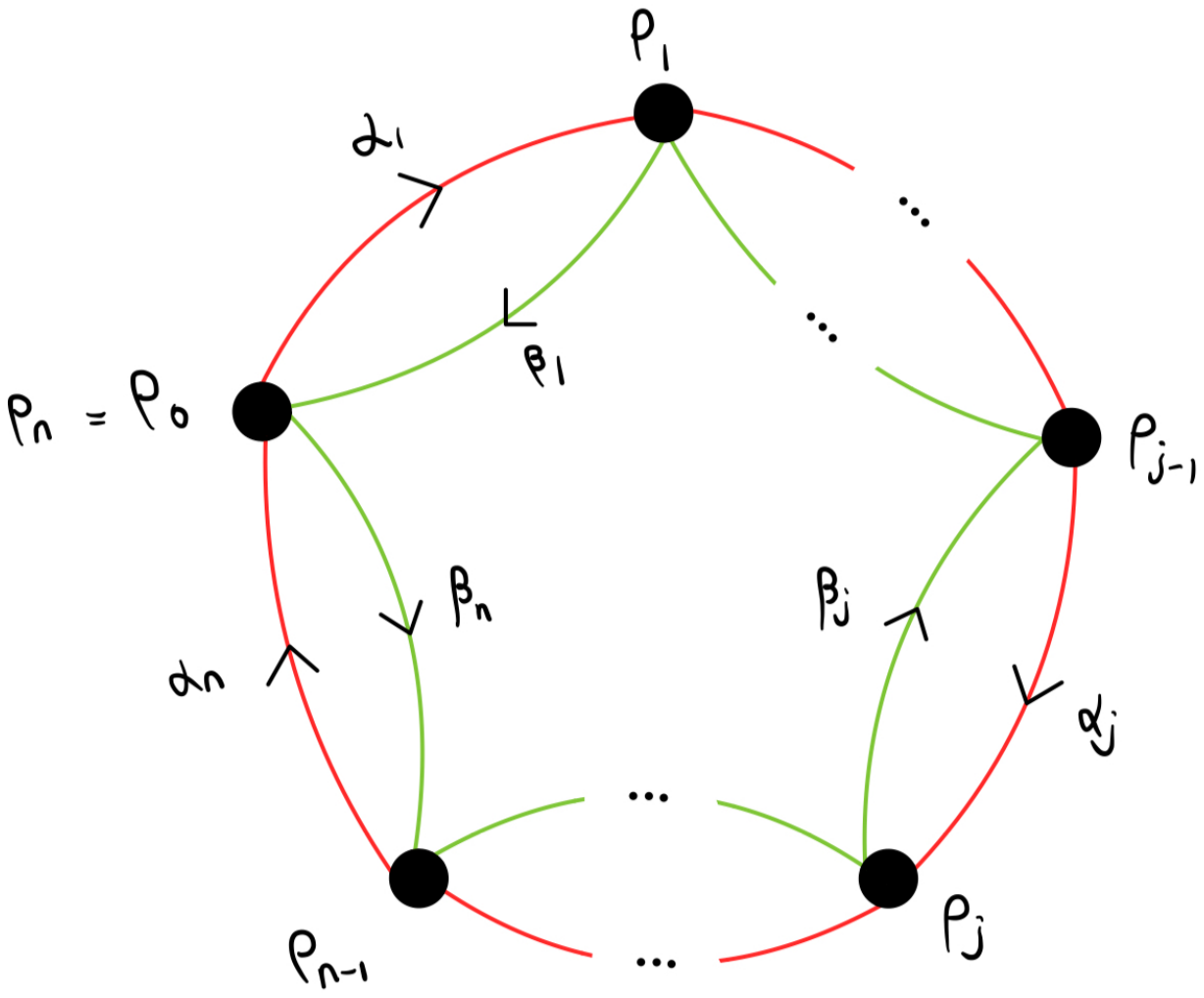}
\captionof{figure}{$(\alpha,\beta)=((\alpha_1,...,\alpha_n),(\beta_1,...,\beta_n))$}
\label{fig:myfigure}
\end{center} This diagram is also known as the extended $A_n$-type Dynkin diagram. Here, the vertices correspond to the irreducible representations of $\Gamma$, denoted by $\rho_j$, where $j$ is an index in $\Z_n$. We have $$(\alpha,\beta)=((\alpha_1,...,\alpha_n),(\beta_1,...,\beta_n)),$$ where $\alpha_j$, $\beta_j$ are the edge maps connecting vertices $\rho_{j-1}$ and $\rho_j$. We let $\alpha_j^*$, $\beta_j^*$ denote the conjugate transposes of $\alpha_j$ and $\beta_j$; note that $\alpha_j^*$, $\beta_j^*$ reverse the arrows of  $\alpha_j$ and $\beta_j$.

To describe the $S^1$-fixed points for $\widetilde{\C^2/\Gamma}$, recall that an $S^1$-fixed point is a pair of $(\alpha,\beta)$ such that $$(e^{i\varphi}\alpha,e^{i\varphi}\beta)=(f(\varphi)\alpha f(\varphi)^{-1},f(\varphi)\beta f(\varphi)^{-1}),$$ where $f(\varphi)=(f_1(\varphi),...,f_n(\varphi))\in F$.  \begin{center}
\includegraphics[width=0.65\textwidth]{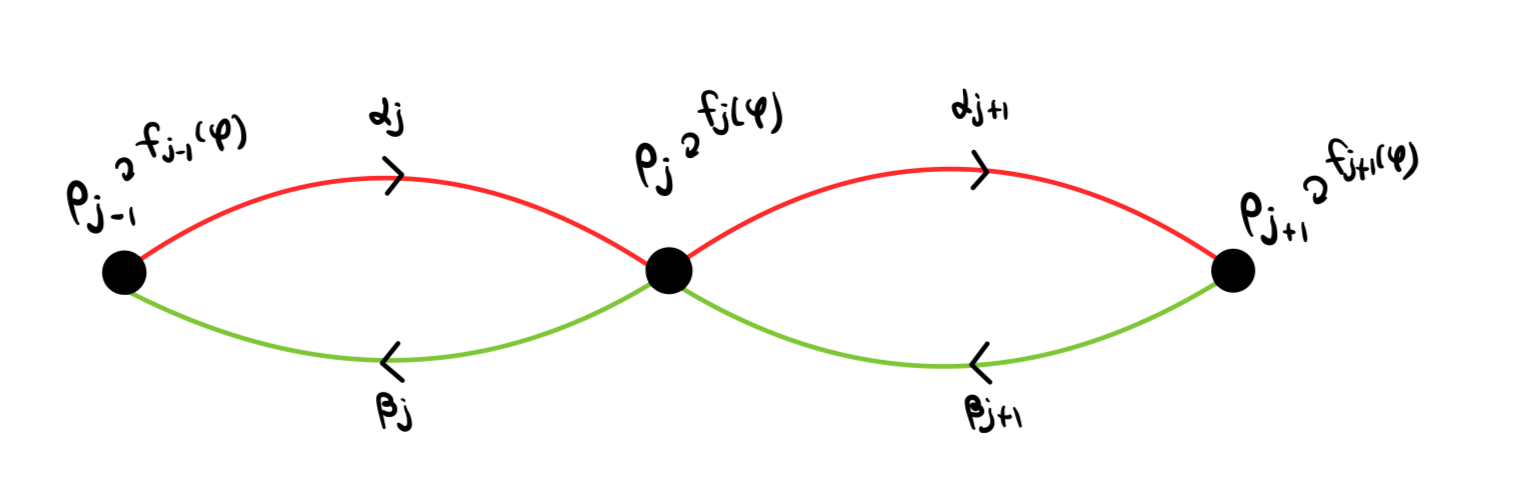}
\captionof{figure}{$f(\varphi)$-action}
\label{fig:1}
\end{center}

From the above diagram, we first observe that $\alpha_j$ and $\beta_j$ cannot both be nonzero, as we always have $$f(\varphi)\alpha_j f(\varphi)^{-1}=f_{j-1}(\varphi)f_{j}^{-1}(\varphi)\alpha_j=e^{w_ji\varphi}\alpha_j,$$$$f(\varphi)\beta_j f(\varphi)^{-1}=f_{j-1}^{-1}(\varphi)f_{j}(\varphi)\beta_j=e^{-w_ji\varphi}\beta_j,$$ that is, $S^1$ acts on $\alpha_j$ and $\beta_j$ with opposite weights $w_j$ and $-w_j$ due to the opposite directions of the arrows. Since the opposite weights cannot both be equal to $1$, we must have at least one of $\alpha_j$ and $\beta_j$ equal to $0$ at an $S^1$-fixed point, for all $j$.

On the other hand, we have that $$[\alpha,\alpha^*]+[\beta,\beta^*]=\zeta_1,$$ $$[\alpha,\beta]=0.$$ These equations implies that \begin{equation} \alpha_j\beta_j = \alpha_k\beta_k, \forall j,k \end{equation}

\begin{equation}
\alpha_{j+1}\alpha_{j+1}^* -\beta_{j+1}\beta_{j+1}^*+\beta_j\beta_j^*-\alpha_j\alpha_j^*=\zeta_1^j, \forall j.
\end{equation}
We remark that there is a slight abuse of notations: $\alpha_j$, $\alpha_j^*$, $\beta_j$, $\beta_j^*$ above refer to the numerical values of the respective maps, but the directions of the maps are suppressed. Notice that when $(\alpha,\beta)$ is an $S^1$-fixed point, we always have $\alpha_j\beta_j=0$, for all $j$. For simplicity, we assume that $\zeta_1=(\zeta_1^0=\zeta_1^n,...,\zeta_1^{n-1})$, with  $$0<\zeta_1^0=\zeta_1^{n}<\zeta_1^1<...<\zeta_1^{n-2},$$ and $$\zeta_1^{n-1}=-\sum_{j=0}^{n-2}\zeta_1^j.$$ 

\subsubsection{$n=2m+1$:}
 
 When $\Gamma$ is an odd cyclic subgroup of $SU(2)$ with $|\Gamma|=n=2m+1$, $\widetilde{\C^2/\Gamma}$ has $n$ isolated $S^1$-fixed points indexed by $j$, with the property that $\alpha_j=\beta_j=0$. It's evident to see that the arrow directions and the edge values are completely determined up to gauge equivalence, once $j$ is fixed. The reason for the existence of some $j$ such that both $\alpha_j$ and $\beta_j$ are $0$ is that the restriction of an $S^1$-fixed point. Indeed, in order to be an $S^1$-fixed point, we must have  $$f(\varphi)\alpha_jf(\varphi)^{-1}= f_{j-1}(\varphi)f_{j}^{-1}(\varphi)\alpha_j=e^{i\varphi}\alpha_j,$$$$f(\varphi)\beta_jf(\varphi)^{-1}=f_{j-1}^{-1}(\varphi)f_{j}(\varphi)\beta_j=e^{i\varphi}\beta_j,$$ for all $j$. This condition forces $f(\varphi)$ to act by weights differing by $1$ at adjacent vertices with the arrow pointing to the lower weight. We see that if the number of vertices is odd, for this condition to hold at all vertices, there must be some j such that $\alpha_j=\beta_j=0.$ On the other hand, there can be at most such a $j$ since in order for (5.17) to hold. It's not hard to see that there is a unique global minimal which is given by the critical point with the same number of $\alpha$-edges and $\beta$-edges, and hence, it must be the unique index-$0$ critical point. Since $\Phi$ is perfect, the rest of the critical points must all be of index-$2$.

\begin{figure}[htbp]
  \centering
  \begin{subfigure}{0.3\textwidth}
    \centering
    \includegraphics[width=\linewidth]{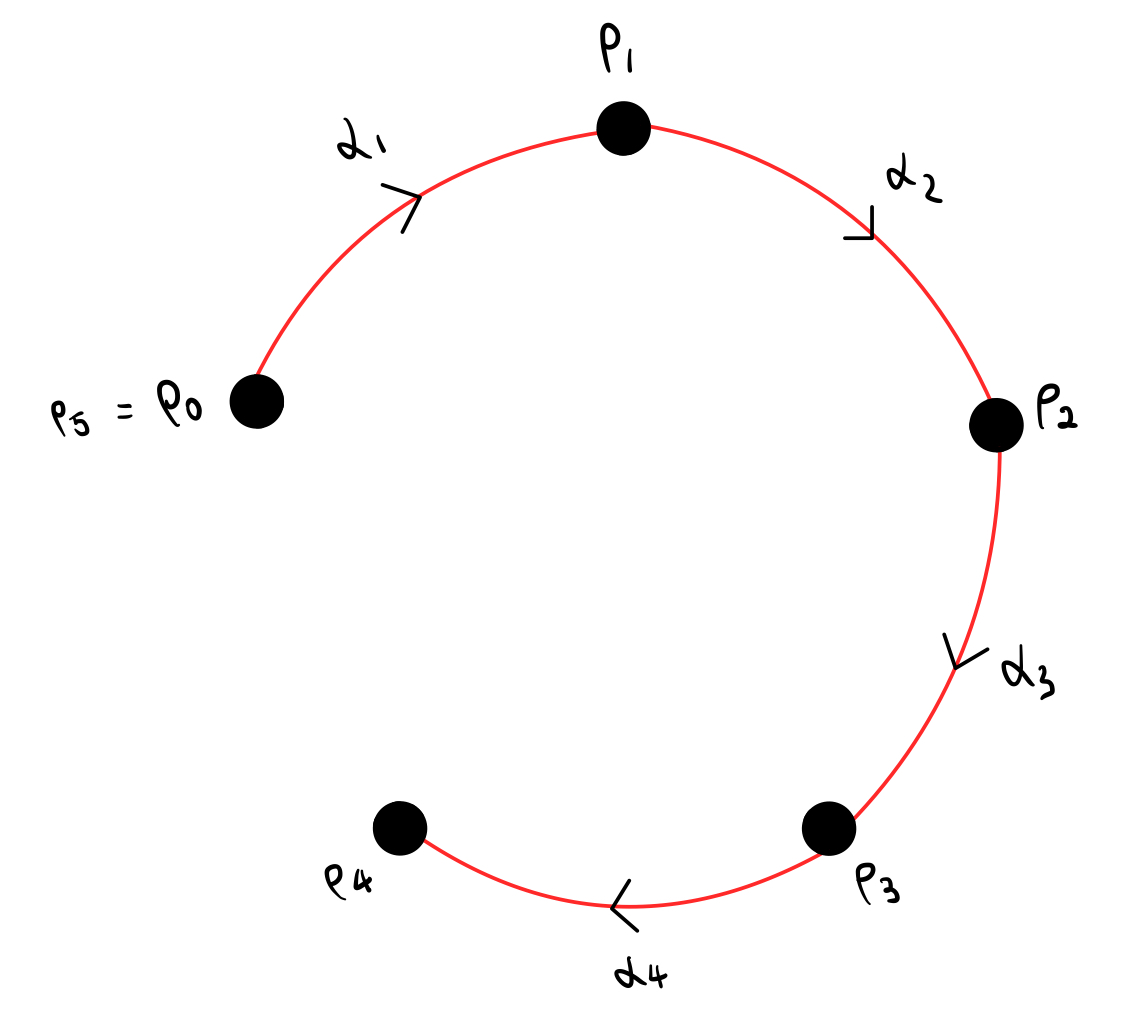}
    \caption{index-$2$ critical point}
    \label{fig:2}
  \end{subfigure}
  \begin{subfigure}{0.3\textwidth}
    \centering
    \includegraphics[width=\linewidth]{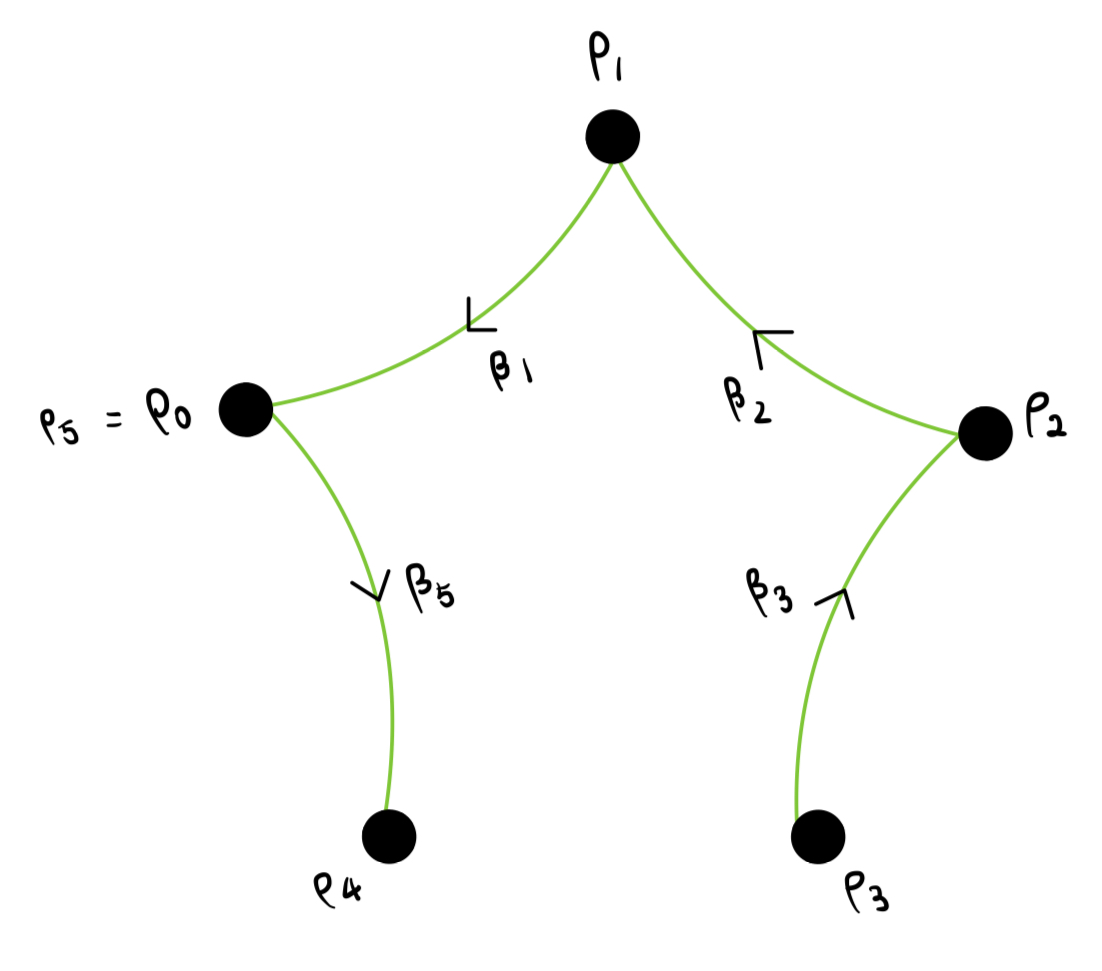}
    \caption{index-$2$  critical point }
    \label{fig:3}
  \end{subfigure}
  \caption{$\Z_5$: $S^1$-fixed points / critical points}
  \label{fig:4}
\end{figure}

  \begin{figure}[htbp]
  \centering
  \begin{subfigure}{0.3\textwidth}
    \centering
    \includegraphics[width=\linewidth]{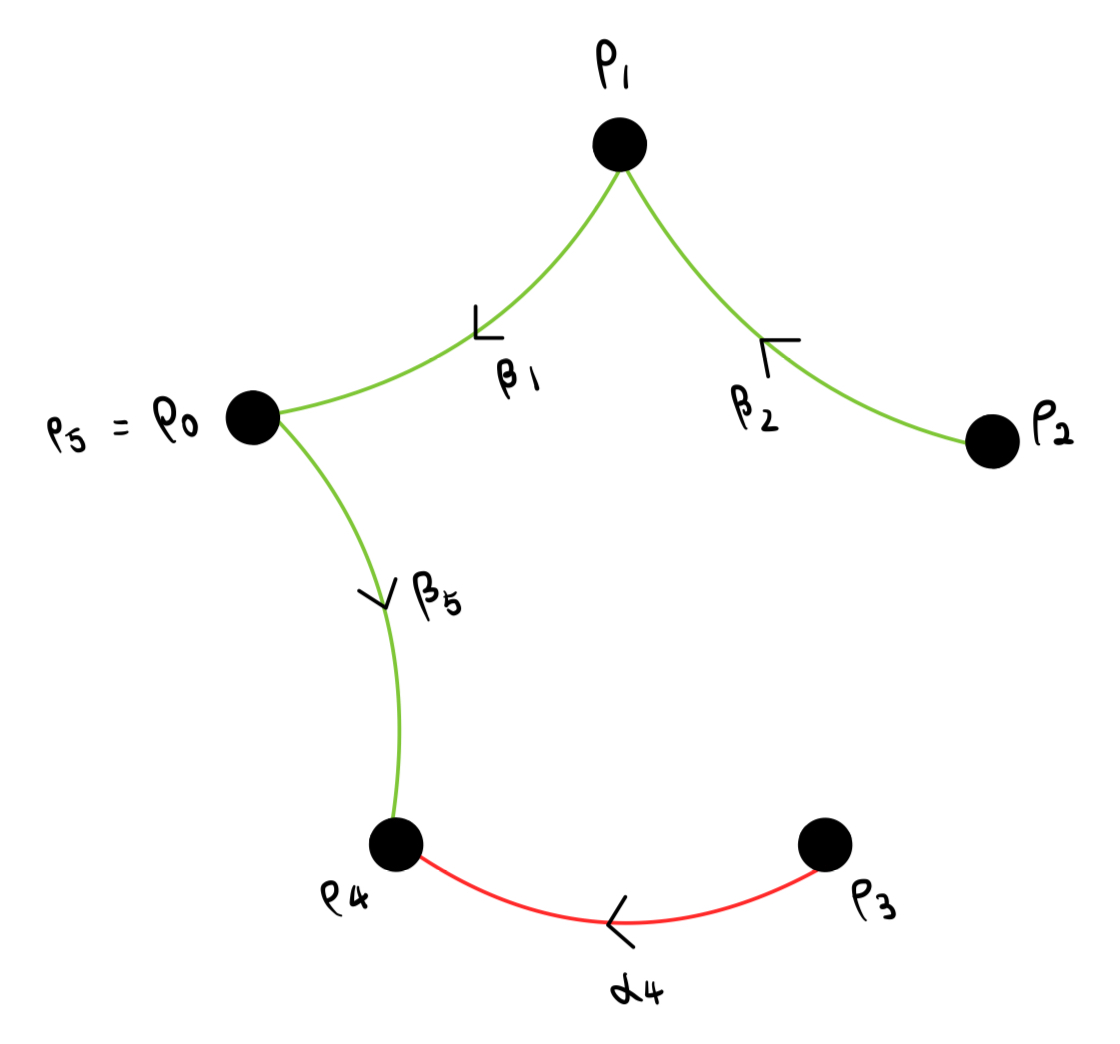}
    \caption{index-$2$  critical point }
    \label{fig:5}
  \end{subfigure}
   \begin{subfigure}{0.3\textwidth}
   
    \centering
    \includegraphics[width=\linewidth]{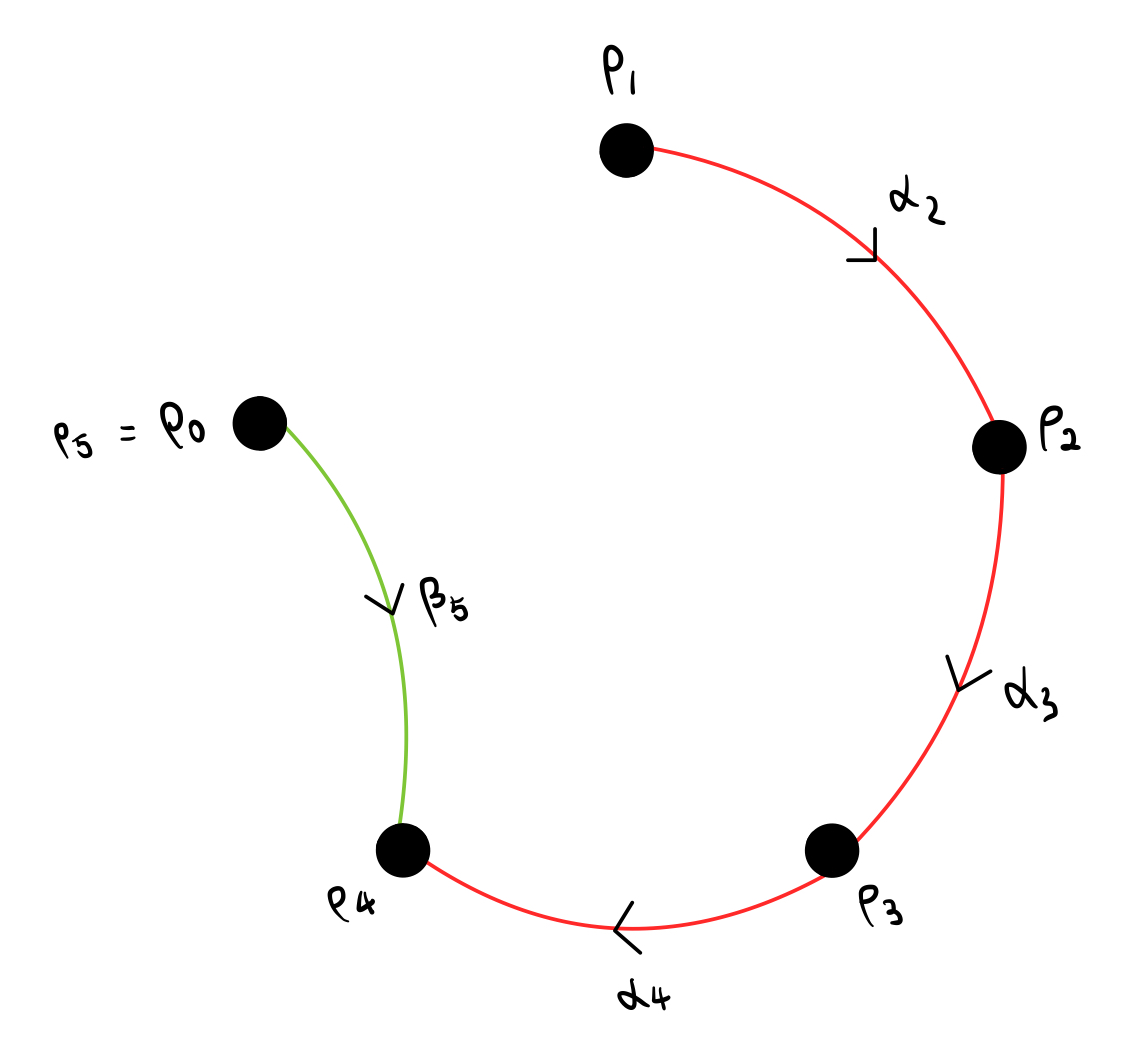}
    \caption{index-$2$  critical point }
    \label{fig:6}
  \end{subfigure}

  \begin{subfigure}{0.3\textwidth}
    \centering
    \includegraphics[width=\linewidth]{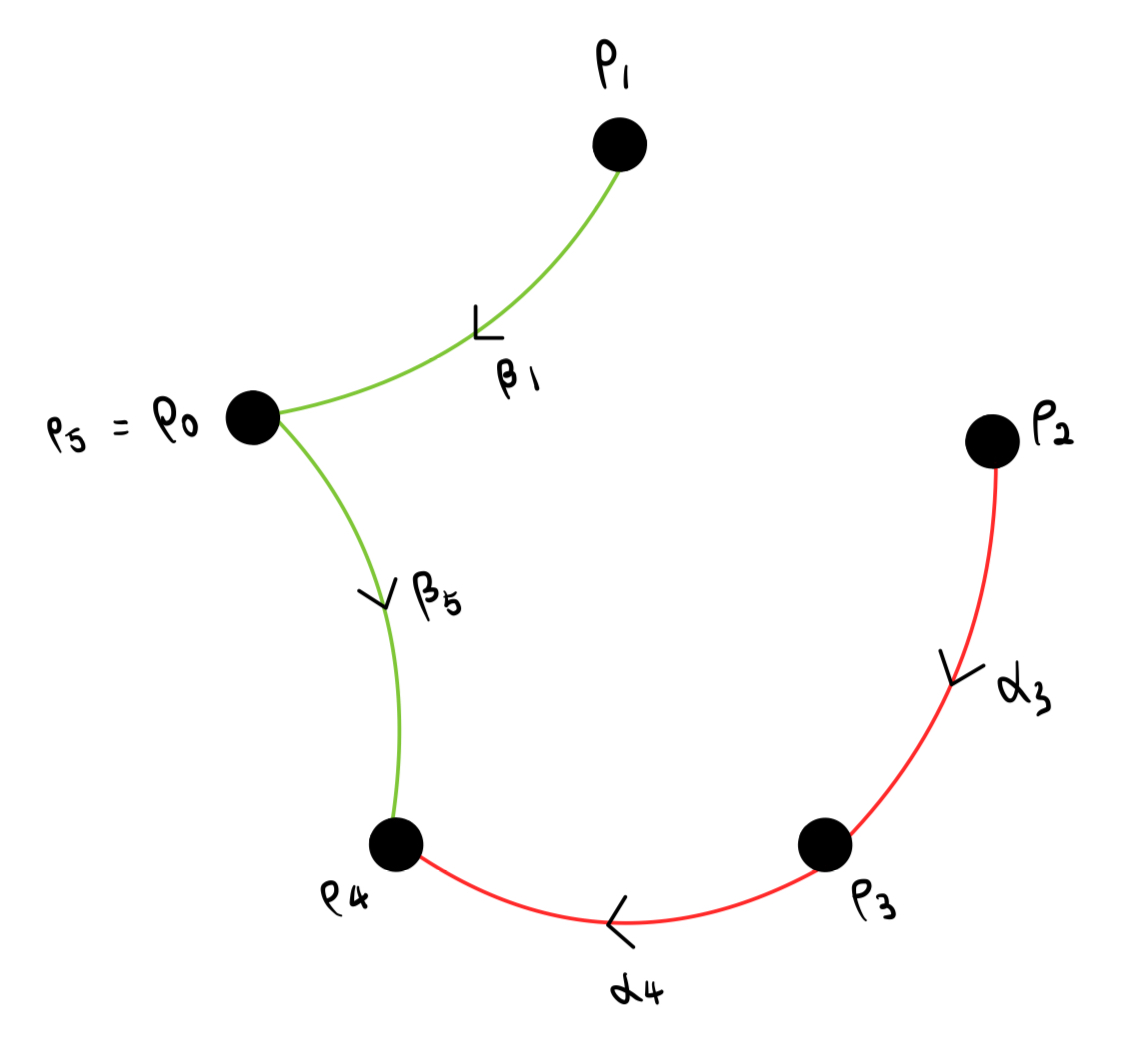}
    \caption{index-$0$  critical point }
    \label{fig:7}
  \end{subfigure}

    \caption{$\Z_5$: $S^1$-fixed points / critical points}
  \label{fig:8}
\end{figure}

\subsubsection{$n=2m$:}

When $\Gamma$ is an even cyclic subgroup of $SU(2)$ of order $|\Gamma|=n=2m$, $\widetilde{\C^2/\Gamma}$ has $n-2$ isolated index-$2$ critical points and a single component of index-$0$ critical submanifold diffeomorphic to a sphere. The index-$0$ critical submanifold is given by the diagram with $m$ $\alpha$-edges and $m$ $\beta$-edges, and there is a sphere of such $S^1$-fixed points up to gauge equivalence with the $\alpha$- and $\beta$-edges taking different norms. It's straightforward to see that $\Phi$ achieves global minimal on this sphere so it is of index-$0$. The rest of the index-$2$ critical points arise in a similar fashion as the odd cyclic case. 

 \begin{figure}[htbp]
  \centering
 
   \begin{subfigure}{0.3\textwidth}
   
    \centering
    \includegraphics[width=\linewidth]{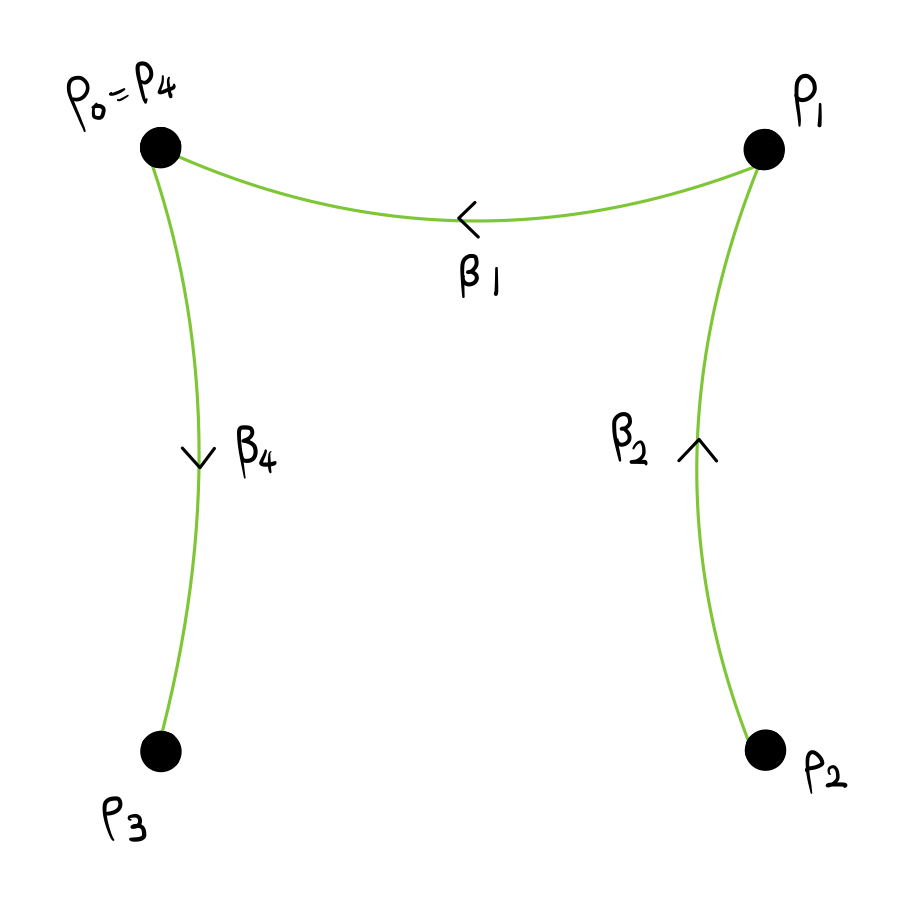}
    \caption{index-$2$ critical point}
    \label{fig:9}
  \end{subfigure}
  \begin{subfigure}{0.3\textwidth}
    \centering
    \includegraphics[width=\linewidth]{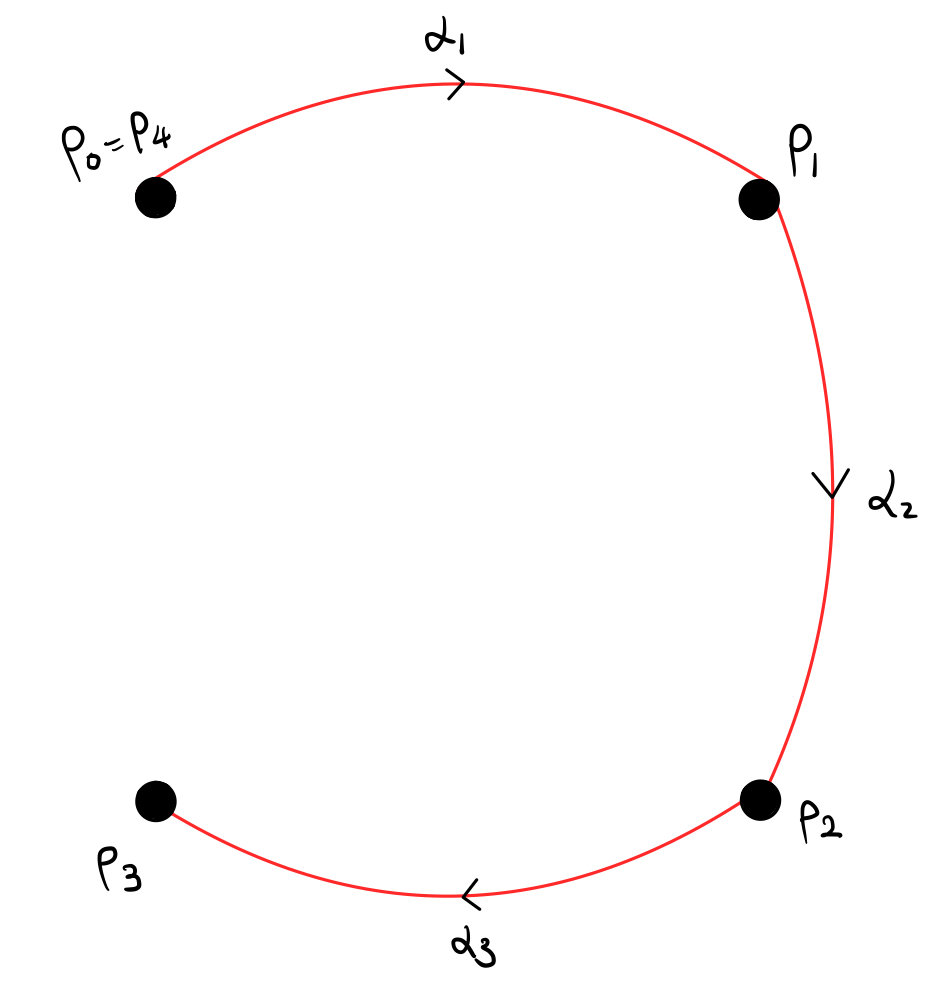}
    \caption{index-$2$ critical point}
    \label{fig:10}
  \end{subfigure}
 
 \begin{subfigure}{0.3\textwidth}
    \centering
    \includegraphics[width=\linewidth]{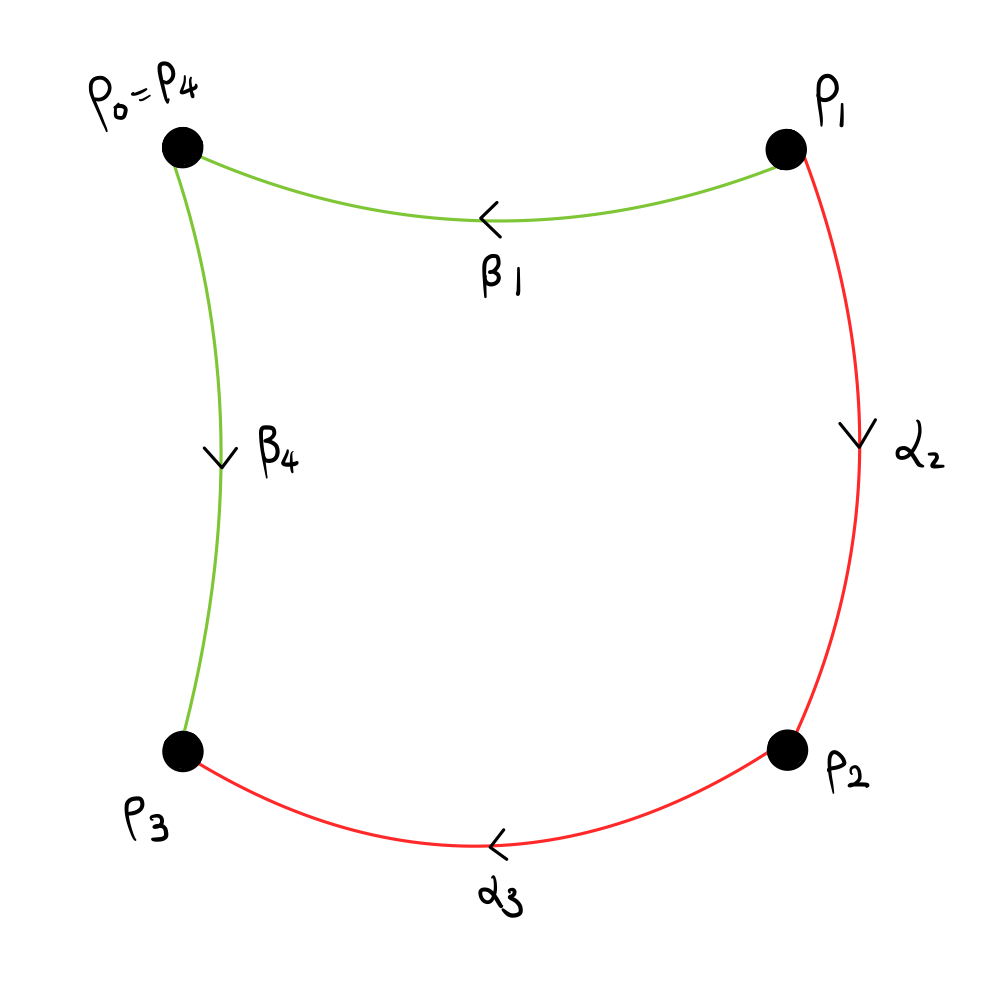}
    \caption{index-$0$ critical submanifold}
    \label{fig:11}
  \end{subfigure}
  
  \caption{$\Z_4$: $S^1$-fixed points / critical points}
  \label{fig:12}
\end{figure}

\subsection{Proof of Proposition \ref{prop}}

We now give the proof of Proposition \ref{prop}. We remark that Proposition \ref{prop} holds for a general $\Gamma$ which we do not assume to be cyclic. 

\begin{proof}[Proof of Proposition \ref{prop}]
We first prove the proposition for $\Theta$ an $S^1$-fixed point.  Recall that for $\Theta=(\alpha,\beta)$, and $\gamma=\begin{pmatrix}
u & v \\
-v^* & u^*  
\end{pmatrix},$ we have the following identities: 
\begin{align}
R(\gamma^{-1})\alpha R(\gamma) &= u\alpha + v\beta, \\
R(\gamma^{-1})\beta R(\gamma)  &= -v^*\alpha + u^*\beta. \label{eq:conj-beta}
\end{align}
It follows from Schur's lemma that $\alpha$, $\beta$ will decompose into edge maps along the McKay quiver of $\Gamma$. Except for when $\Gamma$  is the trivial group where the proposition holds trivially, the McKay quiver of $\Gamma$ contains no self loops at any vertices. As a result, $\alpha$, $\beta$ both have vanishing diagonal blocks as matrices in $End(R)$, where the blocks are given by the irreducible representations of $\Gamma$. Now, if $(\alpha,\beta)$ is an $S^1$-fixed point, there can be at most a single nonzero edge connecting any two vertices. Recall, the holonomy representation is given by $\rho_\Theta=\rho_{(\alpha,\beta)}: \Gamma \to U(R)$, $$\gamma \mapsto R(\gamma)\exp((v_1-(uv_1-\bar{v}v_2))\alpha+(v_2-(\bar{u}v_2+vv_1))\beta).$$ We now focus on the term $\exp((v_1-(uv_1-\bar{v}v_2))\alpha+(v_2-(\bar{u}v_2+vv_1))\beta)$ appearing in the expression. First, we note that a term of the form $\alpha^k\beta^\ell$ can be understood via paths along the directed edges.   Since there can be at most a single nonzero edge connecting any two vertices, by following any path defined by edge maps with the prescribed arrows, one can never come back to the same vertex. In particular, $\alpha$, $\beta$ must be nilpotent matrices, which is a well-known fact for $S^1$-fixed points. This observation implies that $\exp((v_1-(uv_1-\bar{v}v_2))\alpha+(v_2-(\bar{u}v_2+vv_1))\beta)$ also has vanishing diagonal blocks. Hence, $R(\gamma)\exp((v_1-(uv_1-\bar{v}v_2))\alpha+(v_2-(\bar{u}v_2+vv_1))\beta)$ must have the same diagonal blocks as $R(\gamma)$. This implies that $\rho_\Theta$ has the same character as the regular representation $R$ which means they are isomorphic.

For a generic point $\Theta\in X_\zeta$,  the downward gradient flow connecting $\Theta$ to an $S^1$-fixed point gives rise to a smooth path of holonomy representations converging to the regular representation $R$ at the $S^1$-fixed point. Since $\Gamma$ is a finite group, the character variety of $\Gamma$ is discrete, and hence, the smooth path of representations must be the constant path. This implies that $\rho_\Theta$ must be isomorphic to the regular representation for all points in $X_\zeta$. This concludes the proof of the proposition. 

\end{proof}

 \section{Irreducible representations as limits of gradient flow lines}

In this section, we develop a new dynamical framework for the McKay correspondence via gradient flow arising from the Morse function in the previous section. From Proposition \ref{prop}, we know that at every point $(\alpha,\beta)$ of $X_\zeta$, we get a holonomy representation $\rho_(\alpha,\beta)$ isomorphic to the regular representation. Hence, from an index-$2$  $S^1$-fixed point, along a gradient flow line to the boundary at infinity, we get a $1$-parameter family of representations $\rho(t)$ all isomorphic to the regular representation $R$. Hence, we can express each $\rho(t)$ as $$\rho(t)=P(t)RP(t)^{-1},$$ for a $1$-parameter family of change of basis matrices $P(t)$. We recall Conjecture \ref{conj} from Section 1:

\begin{conj}[c.f. Conjecture  \ref{conj}]
Let $\Gamma$ be a finite subgroup of $SU(2)$. Let $X_\zeta$ be as in Theorem \ref{thm}. Let $S_1$,...$S_r$ be the components of $S^1$-fixed points of $X_\zeta$ generating $H^2(X_\zeta)$, and let $\rho_j(t)=P_j(t)RP_j(t)^{-1}, 1\leq j\leq r$ be the $1$-parameter family of holonomy representations with the property that $$\rho_j(t)=\rho_{(\alpha_j(t),\beta_j(t))},$$ $$\lim_{t\to-\infty}(\alpha_j(t),\beta_j(t))=(\alpha_j,\beta_j),$$ with $(\alpha_j,\beta_j)\in S_j$. Then $$P_j^{\lim}=\lim_{t\to\infty}P_j(t)$$ is equal to the projector onto the nontrivial irreducible representation $\rho_j$ of $\Gamma$, and every nontrivial irreducible representation of $\Gamma$ arises this way. 
\end{conj}

In the remaining of this section, we give computations verifying the conjecture for  low rank cyclic cases $\Z_2$, $\Z_3$. The method is indicative of the general verification for all $A_n$-type ALE spaces, which will be contained in a subsequent version of this paper.

\subsection{$\Z_2$:} We let $\rho_0$ denote the trivial representation of $\Z_2$, and $\rho_1$ the nontrivial irreducible representation. We let $\alpha=(\alpha_1,\alpha_2)$ where $\alpha_1$ is the edge map from $\rho_0$ to $\rho_1$ and $\alpha_2$ the edge map from $\rho_1$ to $\rho_0$, and $\alpha_1^*$ is the conjugate transpose of $\alpha_1$ from $\rho_1$ to $\rho_0$, and $\alpha_2^*$ the conjugate transpose of $\alpha_2$ from $\rho_1$ to $\rho_0$. Similarly, we write $\beta=(\beta_1,\beta_2)$ with $\beta_1$ the edge map from $\rho_1$ to $\rho_0$ and $\beta_2$ the edge map from $\rho_0$ to $\rho_1$, and $\beta_1^*$, $\beta_2^*$ are the respective conjugate transposes. Let $\zeta_1=(a^2,-a^2), a\in\R$ be fixed. 
\begin{center}
\includegraphics[width=0.4\textwidth]{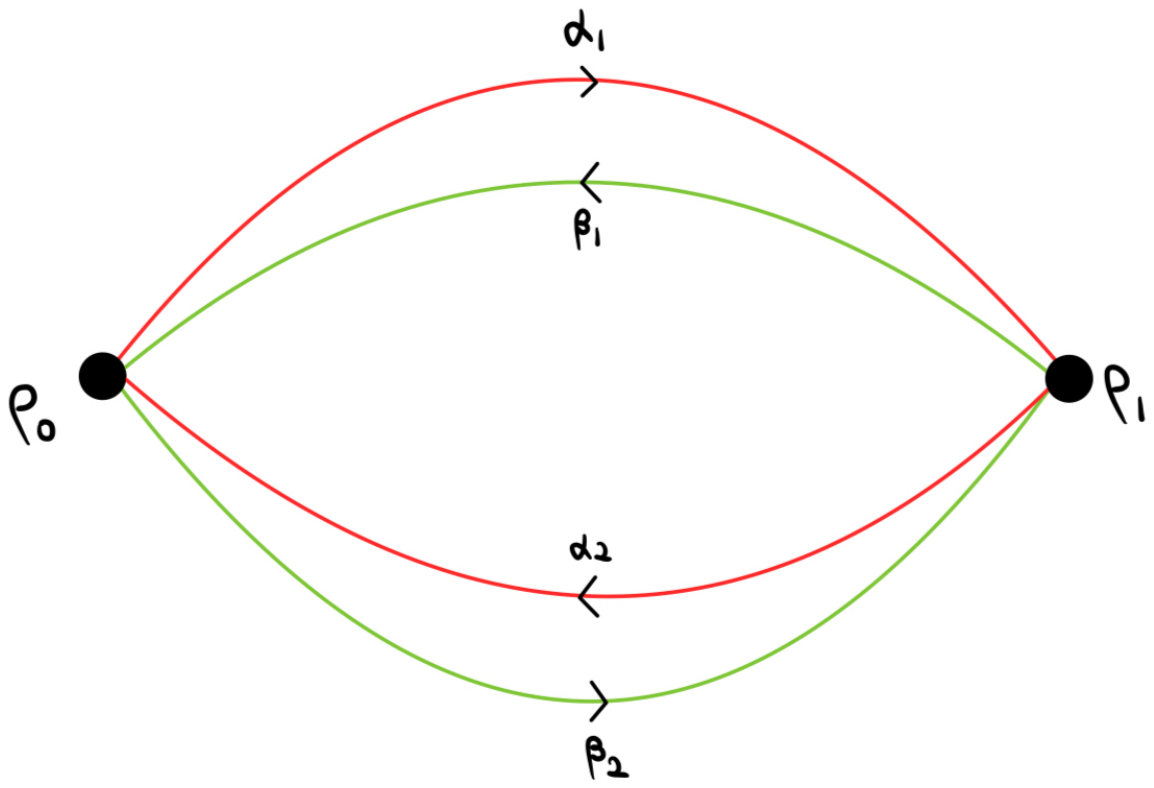}
\captionof{figure}{$(\alpha,\beta)=((\alpha_1,\alpha_2),(\beta_1,\beta_2))$}
\label{fig:13}
\end{center}

The equations \begin{equation} [\alpha,\alpha^*]+[\beta,\beta^*]=\zeta_1 \end{equation}
\begin{equation} [\alpha,\beta]=0 \end{equation} translate to the following equalities: 
\[
\begin{cases}
\alpha_1\beta_1 = \alpha_2\beta_2 \\
(\alpha_1\alpha_1^* -\alpha_2\alpha_2^*)+(\beta_2\beta_2^*-\beta_1\beta_1^*)=a^2.

\end{cases}
\]

There is the same abuse of notations as mentioned in the previous section.

At a critical point, $(\alpha,\beta)$ will be of the form $\alpha_1\alpha_1^*+\beta_2\beta_2^*=a^2$ and $\alpha_2=\beta_1=0$. Up to gauge equivalence, this gives us a single component of $S^1$-invariant critical submanifold diffeomorphic to a sphere. 
\begin{center}
\includegraphics[width=0.4\textwidth]{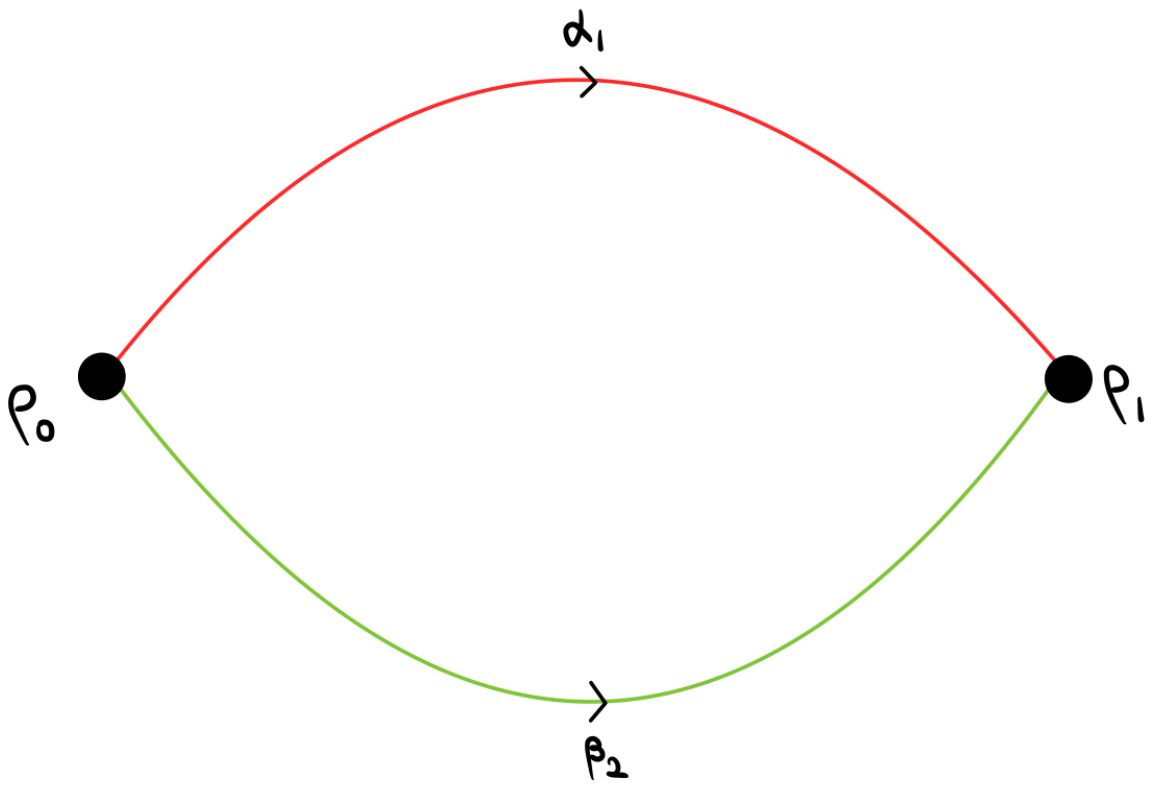}
\captionof{figure}{$S^1$-fixed points / critical points}
\label{fig:14}
\end{center}

Now, we pick any point $x$ on this sphere, and compute $1$-parameter family of the holonomy representations along the gradient flow line from the chosen point $x$ to the boundary at infinity, and study the limit of the given family of holonomy representations at infinity. Specifically, let $(\alpha(t),\beta(t))$ be the gradient flow line such that $(\alpha(-\infty),\beta(-\infty))=x$.  Since $(\alpha(t),\beta(t))$ is a gradient flow line into $x$, we have the following equations satisfied for all $t$:
\begin{equation}2\alpha(t)+[\xi(t),\alpha(t)]=\dot{\alpha}(t), \end{equation}
\begin{equation}2\beta(t)+[\xi(t),\beta(t)]= \dot{\beta}(t), \end{equation}
\begin{equation} [\alpha(t),\alpha(t)^*]+[\beta(t),\beta(t)^*]=\zeta_1, \end{equation}
\begin{equation} [\alpha(t),\beta(t)]=0,\end{equation}
\begin{equation}\lim_{t\to-\infty}(\alpha(t),\beta(t))=x.\end{equation}

Recall, we have $\rho_{\Theta}:\Gamma\to GL(R)$, $$\gamma\mapsto R(\gamma)\exp((v_1-(uv_1-\bar{v}v_2))\alpha+(v_2-(\bar{u}v_2+vv_1))\beta).$$ Any choice of point $x$ on the sphere of $S^1$-fixed points and any choice of generic base point $(v_1,v_2)$ will give rise to the same limiting representation as $t\to\infty$, which we will justify at the end of the subsection. Hence, without loss of generality, we let $$x=(\alpha,\beta)=((\alpha_1,\alpha_2),(\beta_1,\beta_2))=((a,0),(0,0)).$$

  For $\gamma=-1\in SU(2)$, $$\rho(-\infty)(-1)=\rho_\Theta(-1)=R(-1)\exp(2v_1\alpha)$$$$=\begin{pmatrix}
1 & 0 \\
0 & -1  
\end{pmatrix}\exp\begin{pmatrix}
0 & 0	 \\
2v_1a & 0  
\end{pmatrix}=\begin{pmatrix}
1 & 0 \\
0 & -1  
\end{pmatrix}\begin{pmatrix}
1 & 0 \\
2v_1a & 1  
\end{pmatrix}=\begin{pmatrix}
1 & 0 \\
2v_1a & -1  
\end{pmatrix}.$$

Now, consider the following $1$-parameter family $$(\alpha(t),\beta(t))=((\sqrt{e^{c(t)}+a^2},\sqrt{e^{c(t)}}),(0,0)).$$  We can always find a $1$-parameter family of elements $\xi(t)=(\xi_1(t),\xi_2(t))\in\f/\ttt$ and $c(t)$ with $c(-\infty)=-\infty$, $c(\infty)=\infty$, and $\dot{c}(t)>0$, such that $(\alpha(t),\beta(t))$ satisfies equations (6.3)--(6.7) and hence describes a gradient flow line from $x$ to the boundary at infinity. Again, we leave the calculation for $\xi(t)$ and $c(t)$ to the end of the subsection and assume the result for now.

Let $$A(t)=\begin{pmatrix}
0 & 2v_1\sqrt{e^{c(t)}}	 \\
2v_1\sqrt{e^{c(t)}+a^2}& 0  \end{pmatrix}=\begin{pmatrix}
0 & \sqrt{4v_1^2e^{c(t)}}	 \\
\sqrt{4v_1^2e^{c(t)}+4v_1^2a^2}& 0  \end{pmatrix}.
$$
Let $n=4v_1^2e^{c(t)}$ and $\hat{a}=4v_1^2a^2$, and write $$A(n)=\begin{pmatrix}
0 & \sqrt{n}	 \\
\sqrt{n+\hat{a}}& 0  \end{pmatrix}.$$ We have $$\exp(A(n))=\begin{pmatrix}
\cosh(\sqrt[4]{n(n+\hat{a})}) & -\sqrt[4]{\frac{n}{n+\hat{a}}}\sinh(\sqrt[4]{n(n+\hat{a})})	 \\
-\sqrt[4]{\frac{n+\hat{a}}{n}}\sinh(\sqrt[4]{n(n+\hat{a})})& \cosh(\sqrt[4]{n(n+\hat{a})}) \end{pmatrix}.$$ Let $B=R(-1)=\begin{pmatrix}
1 & 0 \\
0 & -1  
\end{pmatrix}$, and $$C(n)=B\exp(A(n))=\begin{pmatrix}
\cosh(\sqrt[4]{n(n+\hat{a})}) & \sqrt[4]{\frac{n}{n+\hat{a}}}\sinh(\sqrt[4]{n(n+\hat{a})}) \\
 -\sqrt[4]{\frac{n+\hat{a}}{n}}\sinh(\sqrt[4]{n(n+\hat{a})}) & -\cosh(\sqrt[4]{n(n+\hat{a})}) 
\end{pmatrix}.$$ It's straightforward to check that $C(n)$ is always conjugate to $B$, which agrees with the prediction that $\rho(t)$ is always isomorphic to the regular representation $R$. Let $P(n)$ be the change of basis matrix such that $P(n)BP(n)^{-1}=C(n)$. By using the following identities of $\cosh(\theta)$ and $\sinh(\theta)$: $$\cosh^2(\theta)-\sinh^2(\theta)=1,$$ $$\cosh(\theta)=\cosh^2(\frac{\theta}{2})+\sinh^2(\frac{\theta}{2}), $$$$\sinh(\theta)=2\cosh(\frac{\theta}{2})\sinh(\frac{\theta}{2}),$$  we see that $P(n)$ is given by $$P(n)= \begin{pmatrix}
\sqrt{k} \cosh\left(\frac{\sqrt[4]{n(n+\hat{a})}}{2}\right) & -\sqrt{k} \sinh\left(\frac{\sqrt[4]{n(n+\hat{a})}}{2}\right) \\
-\frac{1}{\sqrt{k}} \sinh\left(\frac{\sqrt[4]{n(n+\hat{a})}}{2}\right) & \frac{1}{\sqrt{k}} \cosh\left(\frac{\sqrt[4]{n(n+\hat{a})}}{2}\right)
\end{pmatrix},$$ where $k = \sqrt[4]{\frac{n}{n+\hat{a}}}$.

Now, we let $n$ go to infinity and take the limit of $C(n)$ and $P(n)$. After renormalization, we get $$C^{\lim}=\begin{pmatrix}
1 & 1 \\
-1 & -1  
\end{pmatrix},$$ and $$P^{\lim}=\begin{pmatrix}
1 & -1 \\
-1 & 1  
\end{pmatrix}.$$

We see that $P^{\lim}$ is the projector onto the nontrivial irreducible representation of $\Z_2$, as predicted by Conjecture \ref{conj}. 

\subsubsection{Gradient flow analysis}
Now we take more care with justifying the claim that the choice of a point $x$ in the sphere of $S^1$-fixed points doesn't affect the limit $P^{\lim}$, as well as explicitly finding $\xi(t)$ and $c(t)$. Without assuming $\beta_2=0$, a generic $S^1$-fixed point is of the form $$(\alpha,\beta)=((\alpha_1,\alpha_2),(\beta_1,\beta_2))=((r,0),(0,s)),$$ with $r^2+s^2=a^2$. We assume $r,s \neq 0$. For a gradient flow line $(\alpha(t),\beta(t))$ from $((r,0),(0,s))$ to the boundary at infinity, we must have $$(\alpha(t),\beta(t))=((\sqrt{e^{c_1(t)}+r^2},\sqrt{e^{c_1(t)}+e^{c_2(t)}-e^{c_3(t)}},(\sqrt{e^{c_3(t)}},\sqrt{e^{c_2(t)}+s^2})),$$ such that $(\alpha(t),\beta(t))$ solves (6.3)--(6.7) for all $t$. The equations translate to the following equalities: \[
\begin{cases}
 \dot{c}_1(t)= 2(1+r^2e^{-c_1(t)})(2+\xi_1(t)-\xi_2(t))

 \\\dot{c}_2(t)=2(1+s^2e^{-c_2(t)})(2+\xi_1(t)-\xi_2(t))
 \\ \dot{c}_3(t)=2(2+\xi_2(t)-\xi_1(t))
 \\ \dot{c}_1(t)e^{c_1(t)}+\dot{c}_2(t)e^{c_2(t)}-\dot{c}_3(t)e^{c_3(t)}=2(e^{c_1(t)}+e^{c_2(t)}-e^{c_3(t)})(2+\xi_2(t)-\xi_1(t))

\\(e^{c_1(t)}+r^2)e^{c_3(t)}=(e^{c_1(t)}+e^{c_2(t)}-e^{c_3(t)})(e^{c_2(t)}+s^2).
\end{cases}
\]

Let $u(t)=2+\xi_1(t)-\xi_2(t)$, and $2+\xi_2(t)-\xi_1(t)=4-u(t)$. Let $S(t) = e^{c_1(t)} + e^{c_2(t)}$. Then we have $$\dot{c}_1 e^{c_1} + \dot{c}_2 e^{c_2} - \dot{c}_3 e^{c_3} = 2(S - e^{c_3})(4 - u).$$ Further substituting, we get $$2u(S + a^2) - {2e^{c_3}(4 - u)} = 2S(4 - u) - {2e^{c_3}(4 - u)},$$ and hence, $$2u(S + a^2)  = 2S(4 - u) .$$ This gives us $$u(t) = \frac{4S(t)}{2S(t) + a^2}.$$ On the other hand, $$\dot{S}(t) = \dot{c}_1(t)e^{c_1(t)} + \dot{c}_2(t)e^{c_2(t)},$$ which gives us $$\dot{S}(t) = 2u(t)(e^{c_1(t)} + r^2) + 2u(t)(e^{c_2(t)} + s^2)=2u(t)(S(t)+a^2).$$ Substituting $u(t)$, we get an ODE involving $S(t)$: $$\dot{S}(t) = \frac{8S(t)(S(t) + a^2)}{2S(t) + a^2}.$$ Solving for $S(t)$, we get $$S(t)(S(t)+a^2) = C_0\cdot e^{8t},$$ and thus, $$S(t) = \frac{-a^2 + \sqrt{a^4 + 4C_0 e^{8t}}}{2}.$$ Now, we solve for $c_1(t)$, $c_2(t)$ and $c_3(t)$. Notice, we must have $$\frac{e^{c_1(t)} + r^2}{e^{c_2(t)} + s^2} = K,$$ for some constant $K$. Indeed, we have $$\begin{aligned}
\frac{d}{dt} (e^{c_1(t)} + r^2)  &= 2u(t) (e^{c_1(t)} + r^2 ),\\
\frac{d}{dt}(e^{c_2(t)} + s^2)&= 2u(t) (e^{c_2(t)} + s^2),
\end{aligned}$$
and hence, $$\frac{d}{dt} ( \frac{e^{c_1(t)} + r^2}{e^{c_2(t)} + s^2} ) = \frac{2u (e^{c_1} + r^2))(e^{c_2} + s^2) -  (e^{c_1} + r^2)(2u(e^{c_2} + s^2)}{(e^{c_2} + s^2)^2} = 0.$$
Combining everything, we get $$c_1(t) = \ln(\frac{K S(t) + (K s^2 - r^2)}{K + 1}),$$ $$c_2(t) = \ln(\frac{S(t) - (K s^2 - r^2)}{K + 1}),$$ $$c_3(t)= \ln(\frac{S(t)}{K + 1}).$$ To ensure that $c_1(t),c_2(t),c_3(t)$ are defined on $\R$ with the correct boundary conditions, we see that we must have $$K s^2 - r^2=0,$$ which implies $$K=\frac{r^2}{s^2}.$$ Hence, we must have $$c_1(t) = \ln(\frac{K S(t) }{K + 1}),$$ $$c_2(t)=c_3(t)= \ln(\frac{S(t)}{K + 1}).$$ 

In the cases of either $r=0$ or $s=0$, say $s=0$ and $r=a$, we have $c_2=c_3=0$, and  $$c_1(t)=c(t) = \ln ( \frac{-a^2 + \sqrt{a^4 + 4C_0 e^{8t}}}{2}).$$

In all the above cases, the boundary conditions for $t$ are satisfied, and $\xi_1(t), \xi_2(t)$ can always be chosen such that $\xi_1(t)+ \xi_2(t)=0$ to ensure $\xi(t)=(\xi_1(t),\xi_2(t))$ is a traceless element in $\f/\ttt$, for all $t$.

We further observe that regardless of the values of $r$ and $s$, after renormalization, the limiting holonomy representation is unaffected, that is $C^{\lim}$ and  $P^{\lim}$ remain unchanged.

\subsection{$\Z_3$:}
We now study the $\Z_3$ case. Again, we let $\rho_0$ denote the trivial representation of $\Z_2$, and $\rho_1$, $\rho_2$ the nontrivial irreducible representations with eigenvectors $\begin{pmatrix}
1 & \omega &
\omega^2 
\end{pmatrix}$ and $\begin{pmatrix}
1 & \omega^2 &
\omega 
\end{pmatrix}$, respectively. We let $\alpha=(\alpha_1,\alpha_2,\alpha_3)$, $\beta=(\beta_1,\beta_2,\beta_3)$, where $\alpha_j$, $\beta_j$ are again the edge maps, and $\alpha_j^*$, $\beta_j^*$ their conjugate transposes. Let $\zeta_1=(a^2,b^2,-a^2-b^2), a, b \in\R$, $|a|\neq|b|$,  be fixed. 

\begin{center}
\includegraphics[width=0.5\textwidth]{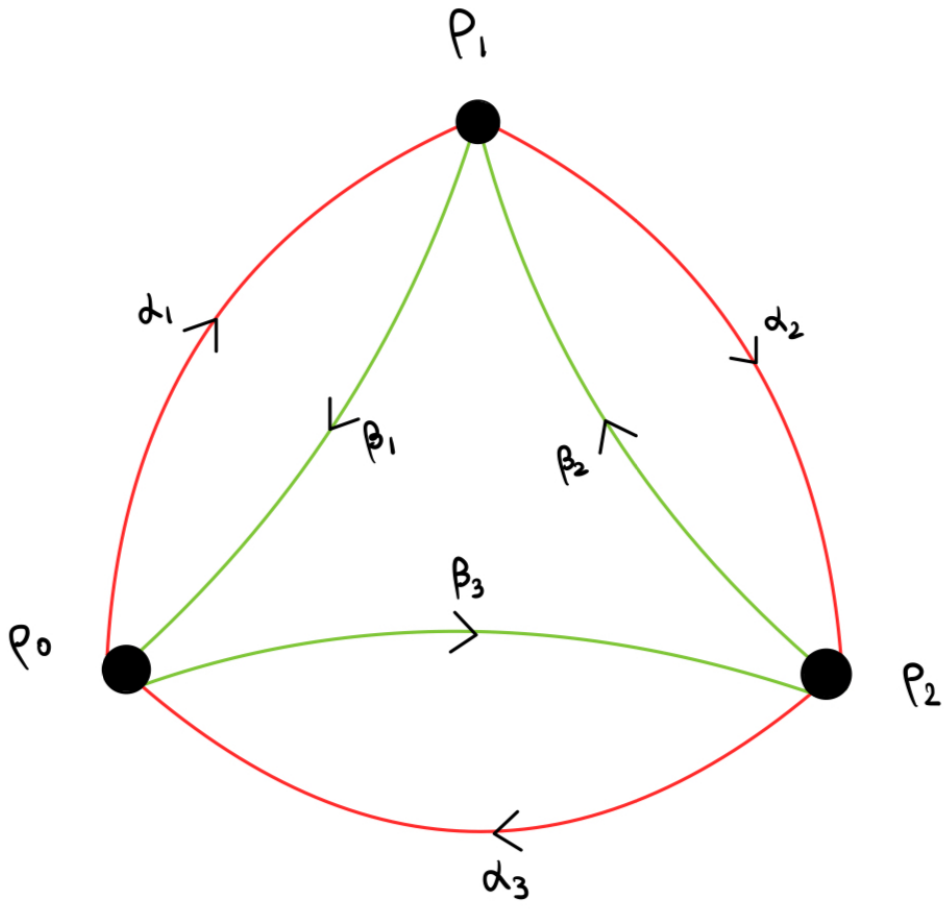}
\captionof{figure}{$(\alpha,\beta)=((\alpha_1,\alpha_2,\alpha_3),(\beta_1,\beta_2,\beta_3))$}
\label{fig:15}
\end{center}

The equations \begin{equation} [\alpha,\alpha^*]+[\beta,\beta^*]=\zeta_1 \end{equation}
\begin{equation} [\alpha,\beta]=0 \end{equation} translate to the following equalities: 
\[
\begin{cases}
\alpha_1\beta_1 = \alpha_2\beta_2= \alpha_3\beta_3 \\
\alpha_1\alpha_1^* -\beta_1\beta_1^*+\beta_3\beta_3^*-\alpha_3\alpha_3^*=a^2 \\
\alpha_2\alpha_2^* -\beta_2\beta_2^*+\beta_1\beta_1^*-\alpha_1\alpha_1^*=b^2.

\end{cases}
\]

Again,  there is the same abuse of notations in the above equations where $\alpha_j$, $\alpha_j^*$, $\beta_j$, $\beta_j^*$  refer to the numerical values of the respective maps with suppressed directions. 

There are three isolated critical points, see Figure 9, up to gauge equivalence, given by 

 \begin{enumerate} 
\item[$(x_1)$] \begin{center}  $\beta_1=\beta_2=\beta_3=\alpha_3=0$, $\alpha_1\alpha_1^*=a^2$, $\alpha_2\alpha_2^*=a^2+b^2,$ \end{center}
\item[$(x_2)$] \begin{center} $\alpha_1=\alpha_2=\alpha_3=\beta_2=0$, $\beta_1\beta_1^*=b^2$, $\beta_3\beta_3^*=a^2+b^2,$ \end{center}
\item[$(x_3)$] \begin{center}$\alpha_1=\beta_1=\beta_2=\alpha_3=0$, $\alpha_2\alpha_2^*=b^2$, $\beta_3\beta_3^*=a^2. $ \end{center}
\end{enumerate} 

\begin{figure}[htbp]
  \centering
  \begin{subfigure}{0.4\textwidth}
    \centering
    \includegraphics[width=\linewidth]{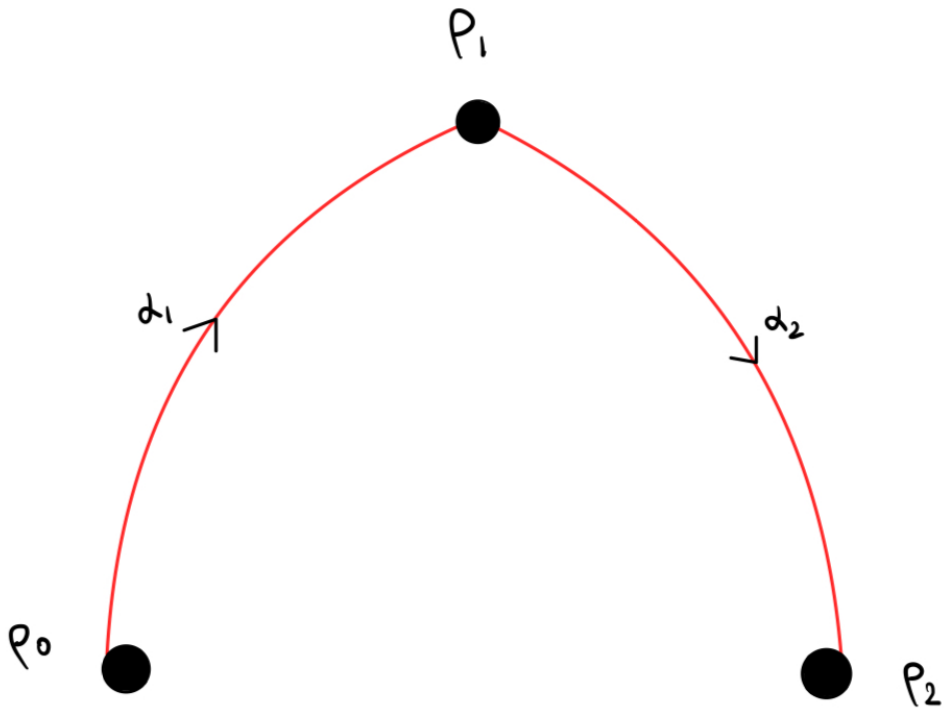}
    \caption{Critical point $x_1$}
    \label{fig:16}
  \end{subfigure}
  \begin{subfigure}{0.4\textwidth}
    \centering
    \includegraphics[width=\linewidth]{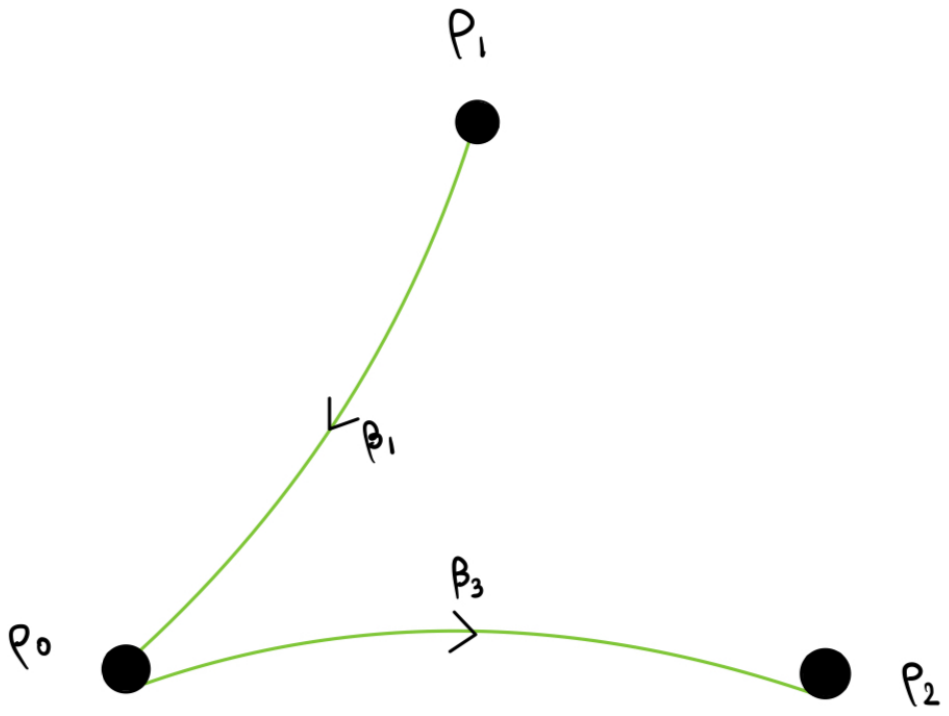}
    \caption{Critical point $x_2$}
    \label{fig:17}
  \end{subfigure}
   \begin{subfigure}{0.4\textwidth}
    \centering
    \includegraphics[width=\linewidth]{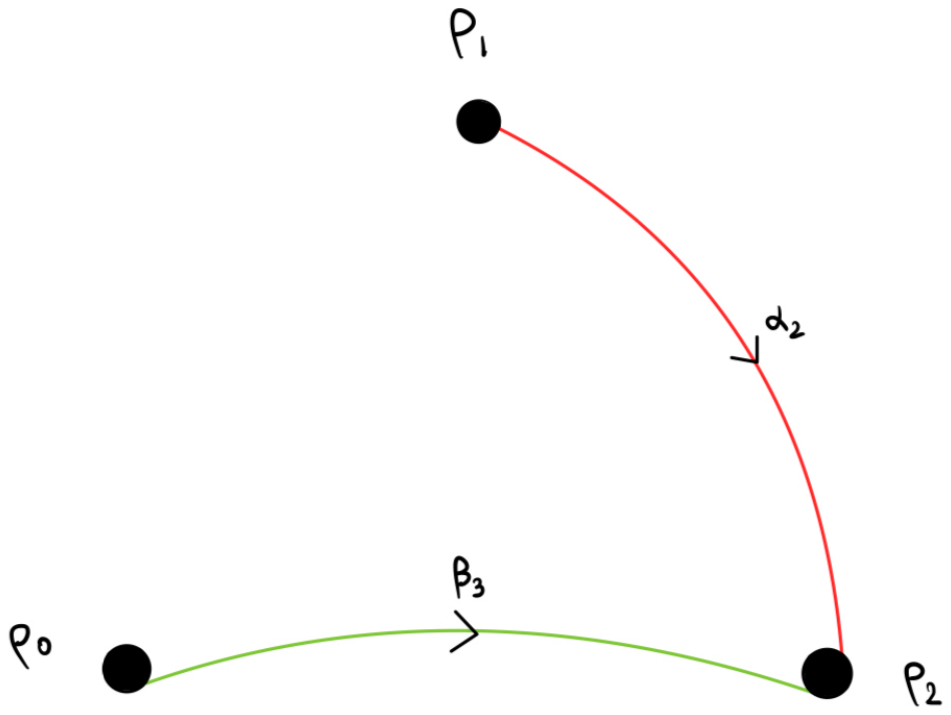}
    \caption{Critical point $x_3$}
    \label{fig:18}
  \end{subfigure}
  \caption{$S^1$-fixed points / critical points}
  \label{fig:19}
\end{figure}

In coordinates, we have: $$x_1=((a,\sqrt{a^2+b^2},0),(0,0,0)),$$
$$x_2=((0,0,0),(b,0,\sqrt{a^2+b^2})),$$ 

$$x_3=((0,b,0),(0,0,a)).$$ 

It's easy to see that  $x_3$ is a global minimal of the Morse function so it must be an index-$0$ critical point, which implies that $x_1$ and $x_2$ are index-$2$ critical points. We now analyze the gradient flow lines from $x_1$ and $x_2$ and calculate the limiting holonomy representations, respectively. 

\subsubsection{Case $x_1$: }
 As before, we start by writing down the gradient flow line $(\alpha(t),\beta(t))$ from $x_1$ to the boundary at infinity, that is, we have the following equations satisfied for all $t$:
\begin{equation}2\alpha(t)+[\xi(t),\alpha(t)]=\dot{\alpha}(t), \end{equation}
\begin{equation}2\beta(t)+[\xi(t),\beta(t)]= \dot{\beta}(t), \end{equation}
\begin{equation} [\alpha(t),\alpha(t)^*]+[\beta(t),\beta(t)^*]=\zeta_1, \end{equation}
\begin{equation} [\alpha(t),\beta(t)]=0,\end{equation}
\begin{equation}\lim_{t\to-\infty}(\alpha(t),\beta(t))=x_1.\end{equation}

Consider the following $1$-parameter family $$(\alpha(t),\beta(t))=((\sqrt{e^{c(t)}+a^2},\sqrt{e^{c(t)}+a^2+b^2},\sqrt{e^{c(t)}}),(0,0,0)).$$ We will show that we can always find $\xi(t)$, a $1$-parameter family of elements in $\f/\ttt$, and a real-valued function $c(t)$ with $c(-\infty)=-\infty $, $c(\infty)=\infty $ and $\dot{c}(t)>0$, such that $(\alpha(t),\beta(t))$ solves equations (6.10)--(6.14). For now, we assume the result, and hence, $(\alpha(t),\beta(t))$ defines a gradient flow line from $x_1$ to infinity.

Now, we calculate the  $1$-parameter family  of holonomy representations $\rho(t)$ along the gradient flow line given by $(\alpha(t),\beta(t))$. Again, $\rho(t):\Gamma\to GL(R)$, $$\gamma\mapsto R(\gamma)\exp((v_1-(uv_1-\bar{v}v_2))\alpha(t)+(v_2-(\bar{u}v_2+vv_1))\beta(t)).$$ We focus on $\gamma_1=\begin{pmatrix}
\omega&0	 \\
0& \omega^{-1}  \end{pmatrix}, $ where $\omega=e^{\frac{2\pi i}{3}}$. Let $N(t)$ be the matrix form of $\alpha(t)$, that is, $$N(t)=\begin{pmatrix}
0 &0	& v_1\sqrt{e^{c(t)}} \\
v_1\sqrt{e^{c(t)}+a^2}& 0 &0 \\0 & v_1\sqrt{e^{c(t)}+a^2+b^2} & 0 \end{pmatrix}. $$ We have 

$$N(t)^2=\begin{pmatrix}
0 &v_1^2\sqrt{e^{c(t)}(e^{c(t)}+a^2+b^2)}	& 0 \\
0& 0 & v_1^2\sqrt{e^{c(t)}(e^{c(t)}+a^2)} \\ v_1^2\sqrt{(e^{c(t)}+a^2)(e^{c(t)}+a^2+b^2)} & 0 & 0 \end{pmatrix},
$$
and $$N(t)^3=v_1^3\sqrt{e^{c(t)}(e^{c(t)}+a^2)(e^{c(t)}+a^2+b^2)}\begin{pmatrix}
1	& 0 &0\\
0& 1&0 \\0 & 0 & 1\end{pmatrix}.
$$

Let  $\eta(t)=(1-\omega)v_1\sqrt[6]{e^{c(t)}(e^{c(t)}+a^2)(e^{c(t)}+a^2+b^2)}$, and let $A_1(t)=(1-\omega)N(t)$. We have $A_1(t)^k=(1-\omega)^kN(t)^k$, and 
\begin{align}
\exp(A_1(t)) &=\sum_{k=0}^\infty\frac{A_1(t)^k}{k!} \\&=I + A_1(t) + \frac{A_1(t)^2}{2!}+ \frac{A_1(t)^3}{3!}+...\\   
 &= (\sum_{j=0}^\infty\frac{A_1(t)^{3j}}{(3j)!} ) + (\sum_{j=0}^\infty\frac{A_1(t)^{3j+1}}{(3j+1)!} )+ (\sum_{j=0}^\infty\frac{A_1(t)^{3j+2}}{(3j+2)!} )\\
    &= (\sum_{j=0}^\infty\frac{\eta(t)^{3j}}{(3j)!} )I+ (\sum_{j=0}^\infty\frac{\eta(t)^{3j+1}}{(3j+1)!} )\frac{A_1(t)}{\eta(t)}+(\sum_{j=0}^\infty\frac{\eta(t)^{3j+2}}{(3j+2)!} )\frac{A_1(t)^2}{\eta(t)^2}.
\end{align}

We now suppress $t$ to simplify the notations. Let $$g_0(\eta)=\sum_{j=0}^\infty\frac{\eta^{3j}}{(3j)!}=\frac{1}{3}(e^\eta+e^{\omega\eta}+e^{\omega^2\eta}),$$
$$g_1(\eta)=\sum_{j=0}^\infty\frac{\eta^{3j+1}}{(3j+1)!}=\frac{1}{3}(e^\eta+\omega^2 e^{\omega\eta}+\omega e^{\omega^2\eta}),$$
$$g_2(\eta)=\sum_{j=0}^\infty\frac{\eta^{3j+2}}{(3j+2)!}=\frac{1}{3}(e^\eta+\omega e^{\omega\eta}+\omega^2 e^{\omega^2\eta}),$$ and we can rewrite $$\exp(A_1)=g_0(\eta)I+\frac{g_1(\eta)}{\eta}A_1+\frac{g_2(\eta)}{\eta^2}A_1^2.$$ Let $c_0=g_0(\eta)$, $c_1=(1-\omega)\frac{g_1(\eta)}{\eta}$ and $c_2=(1-\omega)^2\frac{g_2(\eta)}{\eta^2}$, we have \begin{align}
\exp(A_1(t))&=c_0I+c_1N(t)+c_2N(t)^2\\ 
&=\begin{pmatrix}
c_0 &c_2v_1^2\sqrt{e^{c(t)}(e^{c(t)}+a^2+b^2)}	& c_1v_1\sqrt{e^{c(t)}} \\
c_1v_1\sqrt{e^{c(t)}+a^2}& c_0 &c_2v_1^2\sqrt{e^{c(t)}(e^{c(t)}+a^2)} \\c_2v_1^2\sqrt{(e^{c(t)}+a^2)(e^{c(t)}+a^2+b^2)} & c_1v_1\sqrt{e^{c(t)}+a^2+b^2} & c_0 \end{pmatrix}.
\end{align}

Now, let $$B_1=R(\gamma_1)=\begin{pmatrix}
1	& 0 &0\\
0& \omega&0 \\0 & 0 & \omega^2\end{pmatrix},$$ and \begin{align}C_1(t)&=B_1\exp(A_1(t))\\&=\begin{pmatrix}
c_0 &c_2v_1^2\sqrt{e^{c(t)}(e^{c(t)}+a^2+b^2)}	& c_1v_1\sqrt{e^{c(t)}} \\
\omega c_1v_1\sqrt{e^{c(t)}+a^2}& \omega c_0 &\omega c_2v_1^2\sqrt{e^{c(t)}(e^{c(t)}+a^2)} \\\omega^2 c_2v_1^2\sqrt{(e^{c(t)}+a^2)(e^{c(t)}+a^2+b^2)} & \omega^2 c_1v_1\sqrt{e^{c(t)}+a^2+b^2} & \omega^2c_0 \end{pmatrix}.\end{align}

Again, we now want to find the $1$-parameter family of change of basis matrices $P(t)$ such that $P(t)B_1P(t)^{-1}=C_1(t)$, for all $t$.   We observe that $B_1N=\omega NB_1,$ which implies \begin{equation}B_1A_1=\omega A_1B_1,\end{equation} and hence, we have  $$C_1(t)=\exp(\omega A_1(t))B_1.$$ Meanwhile, by (6.21), we also have $A_1B_1=\omega^2B_1A_1$, and by iterating this identity, we obtain $$A_1^kB_1=\omega^{2k}B_1A_1^k.$$ Let $P^{-1}(t)=\exp(\mu A_1(t)).$ Since $P^{-1}C_1P=P^{-1}B_1\exp(A_1)P=B_1$, we have \begin{equation}\exp(\mu A_1)B_1 \exp(A_1)\exp(-\mu A_1)=B_1.\end{equation} We expand the term $\exp(\mu A_1) B_1$ as follows:

$$\exp(\mu A_1) B_1 = \left( \sum_{n=0}^\infty \frac{(\mu A_1)^n}{n!} \right) B_1 = \sum_{n=0}^\infty \frac{\mu^n}{n!} (A_1^n B_1)$$ Since $A_1^kB_1=\omega^{2k}B_1A_1^k, $ we substitute this into the sum:$$\sum_{n=0}^\infty \frac{\mu^n}{n!} (\omega^{2n} B_1 A_1^n).$$ Since $B_1$ is a constant,  we get $$\exp(\mu A_1) B_1= B_1 \sum_{n=0}^\infty \frac{(\mu \omega^2)^n A_1^n}{n!} = B_1 \sum_{n=0}^\infty \frac{(\mu \omega^2 A_1)^n}{n!}=B_1 \exp(\mu\omega^2 A_1).$$ Hence, by combining with (6.20), we have $$B_1\exp(\mu\omega^2 A_1) \exp(A_1)\exp(-\mu A_1)=B_1, $$ which gives rise to the following equation: $$\mu\omega^2+1-\mu=0.$$

This gives us $\mu=\frac{1}{1-\omega^2}$, and hence, we have $$P=\exp((\omega^2-1)^{-1}A_1)=\exp(\omega N).$$ Let $\lambda(t)=v_1\sqrt[6]{e^{c(t)}(e^{c(t)}+a^2)(e^{c(t)}+a^2+b^2)}$, $$d_0=\frac{1}{3}(e^\lambda+e^{\omega\lambda}+e^{\omega^2\lambda}),$$ $$d_1=\frac{1}{3\lambda}(e^\lambda+\omega^2e^{\omega\lambda}+\omega e^{\omega^2\lambda}),$$ $$d_2=\frac{1}{3\lambda^2}(e^\lambda+\omega e^{\omega\lambda}+\omega^2e^{\omega^2\lambda}).$$ In matrix form, we have \begin{align}P(t)&=d_0 I+d_1\omega N(t)+d_2 \omega^2N(t)^2\\&=\begin{pmatrix}
d_0 &\omega^2d_2v_1^2\sqrt{e^{c(t)}(e^{c(t)}+a^2+b^2)}	& \omega d_1v_1\sqrt{e^{c(t)}} \\
\omega d_1v_1\sqrt{e^{c(t)}+a^2}& d_0 &\omega^2 d_2v_1^2\sqrt{e^{c(t)}(e^{c(t)}+a^2)} \\\omega^2 d_2v_1^2\sqrt{(e^{c(t)}+a^2)(e^{c(t)}+a^2+b^2)} & \omega d_1v_1\sqrt{e^{c(t)}+a^2+b^2} & d_0 \end{pmatrix}.\end{align}

Finally, we analyze the limit of $C_1(t)$ and $P(t)$ as $t$ goes to infinity. We see that in all $c_0,c_1,c_2$ and $d_0,d_1,d_2$, the leading term as $t$ goes to infinity is $e^{\lambda(t)}$, and hence, after renormalization, we have $$C_1^{\lim}=\begin{pmatrix}
1 &1 &1 \\
\omega & \omega &\omega \\\omega^2& \omega^2  & \omega^2 \end{pmatrix},$$ and $$P^{\lim}=\begin{pmatrix}
1 &\omega^2	& \omega \\
\omega & 1&\omega^2 \\\omega^2& \omega  & 1 \end{pmatrix}.$$

For $\gamma_2=\begin{pmatrix}
\omega^{-1}&0	 \\
0& \omega  \end{pmatrix}$, the calculation is similar. Let $A_2(t)=(1-\omega^2)N(t)$, and $$B_2=R(\gamma_2)=\begin{pmatrix}
1	& 0 &0\\
0& \omega^2&0 \\0 & 0 & \omega\end{pmatrix}.$$ Analogously, we have $C_2(t)=B_2\exp(A_2(t)),$ and $P(t)B_2P(t)^{-1}=C_2(t).$ As $t$ goes to infinity, after renormalization, we have $$C_2^{\lim}=\begin{pmatrix}
1 &1 &1 \\
\omega^2 & \omega^2 &\omega^2 \\\omega& \omega  & \omega \end{pmatrix},$$ and the same $$P^{\lim}=\begin{pmatrix}
1 &\omega^2	& \omega \\
\omega & 1&\omega^2 \\\omega^2& \omega  & 1 \end{pmatrix}.$$

We see that $P^{\lim}$ is the projector onto the nontrivial irreducible representation $ \begin{pmatrix}
1	& \omega  &\omega^2 \end{pmatrix}^T$ of $\Z_3$, agreeing with Conjecture \ref{conj}. 

We now go back to solving equations (6.10)--(6.14) explicitly. We need to find $c(t)$ and $\xi(t)=(\xi_1(t),\xi_2(t),\xi_3(t))$ with $$\xi_1(t)+\xi_2(t)+\xi_3(t)=0,$$ $$\lim_{t\to-\infty}c(t)=-\infty, $$$$\lim_{t\to\infty}c(t)=\infty, $$ and $$\dot{c}(t)>0,$$ for all $t$. By the construction of $(\alpha(t),\beta(t))$, (6.12)--(6.14) are always satisfied. Hence, we only need to solve for (6.10)--(6.11). Equations (6.10)--(6.11) translate into the following equalities: \[
\begin{cases}
\dot{c}(t) = 2(2+\xi_3(t)-\xi_1(t))\\
\dot{c}(t) = 2(1+a^2e^{-c(t)})(2+\xi_1(t)-\xi_2(t)) \\
\dot{c}(t) = 2(1+(a^2+b^2)e^{-c(t)})(2+\xi_2(t)-\xi_3(t)).

\end{cases}
\]
Hence, we have \begin{equation}\frac{\dot{c}(t) }{2}=2+\xi_3(t)-\xi_1(t),\end{equation} \begin{equation}\frac{\dot{c}(t) }{ 2(1+a^2e^{-c(t)})}=2+\xi_1(t)-\xi_2(t),\end{equation} \begin{equation}\frac{\dot{c}(t)}{2(1+(a^2+b^2)e^{-c(t)})}=2+\xi_2(t)-\xi_3(t).\end{equation}
Adding (6.25)--(6.27) together, we get $$\frac{\dot{c}(t) }{2}+\frac{\dot{c}(t) }{ 2(1+a^2e^{-c(t)})}+\frac{\dot{c}(t)}{2(1+(a^2+b^2)e^{-c(t)})}=6.$$ Hence, we have $$\dot{c}(t)=\frac{12}{1+\frac{1}{1+a^2e^{-c(t)}}+\frac{1}{1+(a^2+b^2)e^{-c(t)}}}.$$ We solve for $c(t)$ implicitly as follows: $$12t+C_0=c(t)+\ln(e^{c(t)}+a^2)+\ln(e^{c(t)}+a^2+b^2).$$ We see that the boundary conditions are satisfied, $\dot{c}(t)>0$ for all $t$, and we can always choose $\xi_1(t)+\xi_2(t)+\xi_3(t)=0$ so that $\xi(t)$ is a traceless matrix in $\f/\ttt$. This concludes the case analysis for $x_1$.

\subsubsection{Case $x_2$: } The case analysis for $x_2$ is very similar to that of $x_1$. We follow the same steps by first  writing down the gradient flow line $(\alpha(t),\beta(t))$ from $x_2$ to the boundary at infinity. In this case, we have the $1$-parameter family  $$(\alpha(t),\beta(t))=((0,0,0), (\sqrt{e^{c(t)}+b^2},\sqrt{e^{c(t)}}, \sqrt{e^{c(t)}+a^2+b^2})).$$ We skip the explicit calculation for $c(t)$ here since it is identical to the previous case. 

Again,  we compute the holonomy representation, $\rho(t):\Gamma\to GL(R)$, $$\gamma\mapsto R(\gamma)\exp((v_1-(uv_1-\bar{v}v_2))\alpha(t)+(v_2-(\bar{u}v_2+vv_1))\beta(t)),$$ along this gradient flow line by focusing on $\gamma_1=\begin{pmatrix}
\omega&0	 \\
0& \omega^{-1}  \end{pmatrix}$ first.

Let $M(t)$ be the matrix form of $\beta(t)$, that is, $$M(t)=\begin{pmatrix}
0 &v_2\sqrt{e^{c(t)}+b^2}	& 0 \\
0& 0 &v_2\sqrt{e^{c(t)}} \\v_2\sqrt{e^{c(t)}+a^2+b^2} & 0 & 0 \end{pmatrix}. $$ We have 

$$M(t)^2=\begin{pmatrix}
0 &0	& v_2^2 \sqrt{e^{c(t)}(e^{c(t)}+b^2)} \\
v_2^2\sqrt{e^{c(t)}(e^{c(t)}+a^2+b^2)}& 0 & 0\\ 0 & v_2^2\sqrt{(e^{c(t)}+b^2)(e^{c(t)}+a^2+b^2)} & 0 \end{pmatrix},
$$
and $$M(t)^3=v_2^3\sqrt{e^{c(t)}(e^{c(t)}+a^2)(e^{c(t)}+a^2+b^2)}\begin{pmatrix}
1	& 0 &0\\
0& 1&0 \\0 & 0 & 1\end{pmatrix}.
$$

Let $D_1(t)=(1-\omega)M(t)$. We have $D_1(t)^k=(1-\omega)^kM(t)^k$. By the same computation as before, we have \begin{align}\exp(D_1(t))&=c_0I+c_1M(t)+c_2M(t)^2\\ 
&=\begin{pmatrix}
c_0 &c_1v_2\sqrt{e^{c(t)}+b^2}	& c_2v_2^2\sqrt{e^{c(t)}(e^{c(t)}+b^2)} \\
c_2v_2^2\sqrt{e^{c(t)}(e^{c(t)}+a^2+b^2)}& c_0 &c_1v_2\sqrt{e^{c(t)}} \\c_1v_2\sqrt{e^{c(t)}+a^2+b^2} & c_2v_2^2\sqrt{(e^{c(t)}+b^2)(e^{c(t)}+a^2+b^2)} & c_0 \end{pmatrix}.
\end{align}

Now again, let $$B_1=R(\gamma_1)=\begin{pmatrix}
1	& 0 &0\\
0& \omega&0 \\0 & 0 & \omega^2\end{pmatrix},$$ and \begin{align}F_1(t)&=B_1\exp(D_1(t))\\&=\begin{pmatrix}
c_0 &c_1v_2\sqrt{e^{c(t)}+b^2}	& c_2v_2^2\sqrt{e^{c(t)}(e^{c(t)}+b^2)} \\
\omega c_2v_2^2\sqrt{e^{c(t)}(e^{c(t)}+a^2+b^2)}& \omega c_0 &\omega c_1v_2\sqrt{e^{c(t)}} \\\omega^2c_1v_2\sqrt{e^{c(t)}+a^2+b^2} & \omega^2c_2v_2^2\sqrt{(e^{c(t)}+b^2)(e^{c(t)}+a^2+b^2)} & \omega^2c_0 \end{pmatrix}.\end{align}

Now, we find the $1$-parameter family of change of basis matrices $W(t)$ such that $W(t)B_1W(t)^{-1}=F_1(t)$, for all $t$. Following the same steps, we get $$W(t)=\exp(\omega M(t))=d_0I+d_1\omega M(t)+d_2\omega^2M(t)^2,$$ which, in matrix form, is as follows: $$W(t)=\begin{pmatrix}
d_0 &\omega d_1v_2\sqrt{e^{c(t)}+b^2}	& \omega^2d_2v_2^2\sqrt{e^{c(t)}(e^{c(t)}+b^2)} \\
\omega^2d_2v_2^2\sqrt{e^{c(t)}(e^{c(t)}+a^2+b^2)}& d_0 &\omega d_1v_2\sqrt{e^{c(t)}} \\\omega d_1v_2\sqrt{e^{c(t)}+a^2+b^2} & \omega^2 d_2v_2^2\sqrt{(e^{c(t)}+b^2)(e^{c(t)}+a^2+b^2)} & d_0 \end{pmatrix}.$$

By taking the limit as $t$ goes to infinity, after renormalization, we get that $$F_1^{\lim}= \begin{pmatrix}
1 &1 &1 \\
\omega & \omega &\omega \\\omega^2& \omega^2  & \omega^2 \end{pmatrix},$$ and $$W^{\lim}=\begin{pmatrix}
1 &\omega	& \omega^2 \\
\omega^2 & 1&\omega \\\omega& \omega^2  & 1 \end{pmatrix}.
$$

As for  $\gamma_2=\begin{pmatrix}
\omega^{-1}&0	 \\
0& \omega  \end{pmatrix}$, we have $D_2(t)=(1-\omega^2)M(t)$, and $$B_2=R(\gamma_2)=\begin{pmatrix}
1	& 0 &0\\
0& \omega^2&0 \\0 & 0 & \omega\end{pmatrix}.$$ Let $F_2(t)=B_2\exp(D_2(t)),$ and $W(t)B_2W(t)^{-1}=F_2(t).$ As $t$ goes to infinity, after renormalization, we have $$F_2^{\lim}=\begin{pmatrix}
1 &1 &1 \\
\omega^2 & \omega^2 &\omega^2 \\\omega& \omega  & \omega \end{pmatrix},$$ and the same $$W^{\lim}=\begin{pmatrix}
1 &\omega	& \omega^2 \\
\omega^2 & 1&\omega \\\omega& \omega^2  & 1 \end{pmatrix}.$$

We see that $W^{\lim}$ is the projector onto the other nontrivial irreducible representation,  $ \begin{pmatrix}
1	& \omega^2  &\omega \end{pmatrix}^T$, of $\Z_3$, again predicted by Conjecture \ref{conj}.

\end{document}